\documentclass[11pt,a4paper]{amsart}
\usepackage{amssymb,amsthm}

\usepackage[truedimen,margin=30truemm]{geometry}

\usepackage[pdftex,colorlinks=true,bookmarks=true,citecolor=blue,linkcolor=blue,pdfauthor={},pdftitle={}]{hyperref}

\theoremstyle{plain}
\newtheorem{theorem}{Theorem}[section]
\newtheorem{lemma}[theorem]{Lemma}

\theoremstyle{remark}
\newtheorem{remark}[theorem]{Remark}

\numberwithin{equation}{section}

\begin{document}

\title[Heat equation in a moving thin domain]{Error estimate for classical solutions to the heat equation in a moving thin domain and its limit equation}

\author[T.-H. Miura]{Tatsu-Hiko Miura}
\address{Graduate School of Science and Technology, Hirosaki University, 3, Bunkyo-cho, Hirosaki-shi, Aomori, 036-8561, Japan}
\email{thmiura623@hirosaki-u.ac.jp}



\begin{abstract}
  We consider the Neumann type problem of the heat equation in a moving thin domain around a given closed moving hypersurface.
  The main result of this paper is an error estimate in the sup-norm for classical solutions to the thin domain problem and a limit equation on the moving hypersurface which appears in the thin-film limit of the heat equation.
  To prove the error estimate, we show a uniform a priori estimate for a classical solution to the thin domain problem based on the maximum principle.
  Moreover, we construct a suitable approximate solution to the thin domain problem from a classical solution to the limit equation based on an asymptotic expansion of the thin domain problem and apply the uniform a priori estimate to the difference of the approximate solution and a classical solution to the thin domain problem.
\end{abstract}

\maketitle
\section{Introduction} \label{S:Intro}

\subsection{Problem settings and main results} \label{SS:In_Prob}
For $t\in[0,T]$, $T>0$, let $\Gamma(t)$ be a given closed moving hypersurface in $\mathbb{R}^n$, $n\geq2$ with unit outward normal vector field $\nu(\cdot,t)$.
Also, let $g_0(\cdot,t)$ and $g_1(\cdot,t)$ be functions on $\Gamma(t)$ such that $g(y,t)=g_1(y,t)-g_0(y,t)\geq c$ for all $y\in\Gamma(t)$ with some constant $c>0$ independent of $t$.
For a sufficiently small $\varepsilon>0$, we define a moving thin domain
\begin{align*}
  \Omega_\varepsilon(t) = \{y+r\nu(y,t) \mid y\in\Gamma(t), \, \varepsilon g_0(y,t)<r<\varepsilon g_1(y,t)\}, \quad t\in[0,T]
\end{align*}
and consider the Neumann type problem of the heat equation (a linear diffusion equation)
\begin{align} \label{E:Heat_MTD}
  \left\{
  \begin{alignedat}{3}
    \partial_t\rho^\varepsilon-k_d\Delta\rho^\varepsilon &= f^\varepsilon &\quad &\text{in} &\quad &Q_{\varepsilon,T}, \\
    \partial_{\nu_\varepsilon}\rho^\varepsilon+k_d^{-1}V_\varepsilon\rho^\varepsilon &= 0 &\quad &\text{on} &\quad &\partial_\ell Q_{\varepsilon,T}, \\
    \rho^\varepsilon|_{t=0} &= \rho_0^\varepsilon &\quad &\text{in} &\quad &\Omega_\varepsilon(0).
  \end{alignedat}
  \right.
\end{align}
Here $Q_{\varepsilon,T}$ and $\partial_\ell Q_{\varepsilon,T}$ are a space-time domain and its lateral boundary given by
\begin{align*}
  Q_{\varepsilon,T} = \bigcup_{t\in(0,T]}\Omega_\varepsilon(t)\times\{t\}, \quad \partial_\ell Q_{\varepsilon,T} = \bigcup_{t\in(0,T]}\partial\Omega_\varepsilon(t)\times\{t\},
\end{align*}
and $\nu_\varepsilon$ and $V_\varepsilon$ are the unit outward normal vector field and the (scalar) outer normal velocity of $\partial\Omega_\varepsilon(t)$.
Moreover, $k_d>0$ is a diffusion coefficient independent of $\varepsilon$ and $f^\varepsilon$ and $\rho_0^\varepsilon$ are a given source term and initial data which may depend on $\varepsilon$.
Note that the term $k_d^{-1}V_\varepsilon\rho^\varepsilon$ is added in the boundary condition so that the conservation law
\begin{align*}
  \frac{d}{dt}\int_{\Omega_{\varepsilon}(t)}\rho^\varepsilon\,dx = \int_{\Omega_\varepsilon(t)}f^\varepsilon\,dx, \quad t\in(0,T)
\end{align*}
holds by the Reynolds transport theorem and integration by parts.
It can be also seen that the heat (or a substance) moves along the boundary so that it does not go out of or come into the moving thin domain through the boundary.

The purpose of this paper is to compare a solution to \eqref{E:Heat_MTD} with a solution to
\begin{align} \label{E:Limit}
  \left\{
  \begin{alignedat}{3}
    \partial^\circ(g\eta)-gV_\Gamma H\eta-k_d\,\mathrm{div}_\Gamma(g\nabla_\Gamma\eta) &= gf &\quad &\text{on} &\quad &S_T, \\
    \eta|_{t=0} &= \eta_0 &\quad &\text{on} &\quad &\Gamma(0),
  \end{alignedat}
  \right.
\end{align}
which we call the limit equation of \eqref{E:Heat_MTD}.
Here
\begin{align*}
  S_T = \bigcup_{t\in(0,T]}\Gamma(t)\times\{t\}
\end{align*}
is a space-time hypersurface, $\partial^\circ=\partial_t+V_\Gamma\nu\cdot\nabla$ is the normal time derivative, $V_\Gamma$ is the (scalar) outer normal velocity of $\Gamma(t)$, $H$ is the mean curvature of $\Gamma(t)$ which is just the sum of the principal curvatures of $\Gamma(t)$, and $\mathrm{div}_\Gamma$ and $\nabla_\Gamma$ are the surface divergence and the tangential gradient on $\Gamma(t)$.
Also, $f$ and $\eta_0$ are a given source term and initial data.

In our previous work \cite{Miu17}, we studied the thin-film limit problem for \eqref{E:Heat_MTD} with $k_d=1$ and $f^\varepsilon\equiv0$ in the $L^2$-framework and derived \eqref{E:Limit} from \eqref{E:Heat_MTD} rigorously by means of the convergence of a solution and the characterization of the limit.
For an $L^2$-weak solution $\rho^\varepsilon$ to \eqref{E:Heat_MTD}, we proved that the average of $\rho^\varepsilon$ in the thin direction converges weakly to a function $\eta$ on $S_T$ in an appropriate function space on $S_T$ as $\varepsilon\to0$.
Moreover, we derived a weak form on $S_T$ satisfied by $\eta$ from the average of the weak form of \eqref{E:Heat_MTD} and obtained the limit equation \eqref{E:Limit} as an equation to which the weak limit $\eta$ is a unique $L^2$-weak solution.
We also estimated the difference of $\eta$ and the average of $\rho^\varepsilon$ on $S_T$ by a standard energy method and used it to get an $L^2$-error estimate in $Q_{\varepsilon,T}$ for $\rho^\varepsilon$ and the constant extension of $\eta$ in the normal direction of $\Gamma(t)$.

The above result shows that a solution to the limit equation \eqref{E:Limit} approximates a solution to the thin domain problem \eqref{E:Heat_MTD} in the $L^2$-framework.
However, the error estimate in the $L^2(Q_{\varepsilon,T})$-norm may have some ambiguity since the volume of $Q_{\varepsilon,T}$ is of order $\varepsilon$.
One approach to avoid such an ambiguity is to divide the $L^2(Q_{\varepsilon,T})$-norm by $\varepsilon^{1/2}$.
In fact, the error estimate obtained in \cite[Theorem 6.12]{Miu17} shows that
\begin{align} \label{E:Error_L2}
  \varepsilon^{-1/2}\|\rho^\varepsilon-\bar{\eta}\|_{L^2(Q_{\varepsilon,T})} \leq c\Bigl(\varepsilon^{-1/2}\|\rho_0^\varepsilon-\bar{\eta}_0\|_{L^2(\Omega_\varepsilon(0))}+\varepsilon\|\eta_0\|_{L^2(\Gamma(0))}\Bigr)
\end{align}
for weak solutions $\rho^\varepsilon$ to \eqref{E:Heat_MTD} and $\eta$ to \eqref{E:Limit} when $k_d=1$, $f^\varepsilon\equiv0$, and $f\equiv0$, where $\bar{\eta}(\cdot,t)$ is the constant extension of $\eta(\cdot,t)$ in the normal direction of $\Gamma(t)$ for each $t\in[0,T]$.
Hence we can say that $\eta$ approximates $\rho_\varepsilon$ of order $\varepsilon$ in the $L^2$-framework.

The purpose of this paper is to give another approach to avoid the ambiguity due to the volume of $Q_{\varepsilon,T}$: an error estimate in the sup-norm.
We estimate the difference of classical solutions to \eqref{E:Heat_MTD} and \eqref{E:Limit} in the sup-norm on $Q_{\varepsilon,T}$.
It seems that such an attempt for a curved thin domain around a hypersurface is first done in this paper even if a curved thin domain does not move in time, although a similar idea has already appeared in the case of a stationary flat thin domain around a lower dimensional domain (see e.g. \cite{MaNaPl00_02}).

To state our main results, we give some definitions and notations.
For given data $\rho_0^\varepsilon$ and $f^\varepsilon$, we say that $\rho^\varepsilon$ is a classical solution to \eqref{E:Heat_MTD} if
\begin{align*}
  \rho^\varepsilon \in C(\overline{Q_{\varepsilon,T}}), \quad \partial_i\rho^\varepsilon \in C(Q_{\varepsilon,T}\cup\partial_\ell Q_{\varepsilon,T}), \quad \partial_t\rho^\varepsilon, \partial_i\partial_j\rho^\varepsilon \in C(Q_{\varepsilon,T})
\end{align*}
for all $i,j=1,\dots,n$ and $\rho^\varepsilon$ satisfies \eqref{E:Heat_MTD} at each point of $\overline{Q_{\varepsilon,T}}$.
Also, for given data $\eta_0$ and $f$, we say that $\eta$ is a classical solution to \eqref{E:Limit} if $\eta\in C(\overline{S_T})\cap C^{2,1}(S_T)$ and $\eta$ satisfies \eqref{E:Limit} at each point of $\overline{S_T}$, where
\begin{align*}
  C^{2,1}(S_T) = \{\zeta\in C(S_T) \mid \text{$\partial^\circ\zeta,\underline{D}_i\zeta,\underline{D}_i\underline{D}_j\zeta\in C(S_T)$ for all $i,j=1,\dots,n$}\}
\end{align*}
and $\underline{D}_i$ is the $i$-th component of $\nabla_\Gamma$ (see Section \ref{SS:FixSurf}).
Here we do not touch the problem of the existence of classical solutions.
It is known that there exists a classical solution to \eqref{E:Heat_MTD} if given data have a sufficient H\"{o}lder regularity (see e.g. \cite{LaSoUr68,Lie96}).
Also, the existence of a classical solution to \eqref{E:Limit} for sufficiently smooth given data can be shown by a standard localization argument and the existence result in the case of a flat domain, although there seems to be no literature giving the procedure explicitly in the case of a moving surface.
We also note that classical solutions to \eqref{E:Heat_MTD} and \eqref{E:Limit} are unique by the maximum principle.

Let us fix some more notations.
Let $\Omega$ be a spacial domain or hypersurface in $\mathbb{R}^n$ or a space-time domain or hypersurface in $\mathbb{R}^{n+1}$.
For a continuous function $\rho$ on $\Omega$ or on the closure $\overline{\Omega}$, we write
\begin{align*}
  \|\rho\|_{C(\Omega)} = \sup_{\Omega}|\rho|, \quad \|\rho\|_{C(\overline{\Omega})} = \sup_{\overline{\Omega}}|\rho|.
\end{align*}
Note that $\overline{\Omega}$ includes the initial time $t=0$ when $\Omega=Q_{\varepsilon,T},S_T$.
Also, we set
\begin{align} \label{E:C21_Norm}
  \|\eta\|_{C^{2,1}(S_T)} = \|\eta\|_{C(S_T)}+\|\partial^\circ\eta\|_{C(S_T)}+\sum_{i=1}^n\|\underline{D}_i\eta\|_{C(S_T)}+\sum_{i,j=1}^n\|\underline{D}_i\underline{D}_j\eta\|_{C(S_T)}
\end{align}
for $\eta\in C^{2,1}(S_T)$.
We take a constant $\delta>0$ such that $\Gamma(t)$ has a tubular neighborhood of radius $\delta$ in $\mathbb{R}^n$ for all $t\in[0,T]$ (see Section \ref{SS:MovSurf}), and fix $\varepsilon_0\in(0,1)$ such that $\varepsilon|g_i|\leq\delta$ on $\overline{S_T}$ for all $\varepsilon\in(0,\varepsilon_0)$ and $i=0,1$.
Moreover, for a function $\zeta$ on $\overline{S_T}$ and each $t\in[0,T]$, we denote by $\bar{\zeta}(\cdot,t)$ the constant extension of $\zeta(\cdot,t)$ in the normal direction of $\Gamma(t)$.

Now let us state our main results.
First we give an error estimate when $\Omega_\varepsilon(t)$ is a thin tubular neighborhood for each $t\in[0,T]$.

\begin{theorem} \label{T:Error_Spe}
  Let $\varepsilon\in(0,\varepsilon_0)$ and
  \begin{align*}
    \rho_0^\varepsilon\in C(\overline{\Omega_\varepsilon(0)}), \quad f^\varepsilon\in C(Q_{\varepsilon,T}), \quad \eta_0\in C(\Gamma(0)), \quad f\in C(S_T).
  \end{align*}
  Also, let $\rho^\varepsilon$ and $\eta$ be classical solutions to \eqref{E:Heat_MTD} and \eqref{E:Limit}.
  Suppose that
  \begin{itemize}
    \item[(a)] $g_0=g_0(t)$ and $g_1=g_1(t)$ depend only on $t\in[0,T]$.
  \end{itemize}
  Then there exists a constant $c_T>0$ depending on $T$ but independent of $\varepsilon$ such that
  \begin{align} \label{E:Error_Spe}
    \begin{aligned}
      \|\rho^\varepsilon-\bar{\eta}\|_{C(\overline{Q_{\varepsilon,T}})} &\leq c_T\Bigl(\|\rho_0^\varepsilon-\bar{\eta}_0\|_{C(\overline{\Omega_\varepsilon(0)})}+\|f^\varepsilon-\bar{f}\|_{C(Q_{\varepsilon,T})}\Bigr) \\
      &\qquad +\varepsilon c_T\Bigl(\|\eta_0\|_{C(\Gamma(0))}+\|\eta\|_{C^{2,1}(S_T)}\Bigr)
    \end{aligned}
  \end{align}
  provided that the right-hand side is finite.
\end{theorem}

Note that we leave the $C^{2,1}(S_T)$-norm of $\eta$ in the right-hand side of \eqref{E:Error_Spe}, which cannot be estimated just by the sup-norms of $\eta_0$ and $f$.
It can be estimated by the H\"{o}lder norms of $\eta_0$ and $f$ if they are sufficiently smooth by a localization argument and a regularity estimate for a classical solution in the case of a flat domain (see e.g. \cite{LaSoUr68,Lie96}).

When $\Omega_\varepsilon(t)$ is not a tubular neighborhood, i.e. $g_0$ and $g_1$ are not constant in the variable $y\in\Gamma(t)$, we require an additional regularity assumption on $\eta$.

\begin{theorem} \label{T:Error_Gen}
  Under the settings of Theorem \ref{T:Error_Spe}, suppose instead of \emph{(a)} that
  \begin{itemize}
    \item[(b)] $h_i=\nabla_\Gamma g_i\cdot\nabla_\Gamma\eta \in C(\overline{S_T})\cap C^{2,1}(S_T)$ for $i=0,1$.
  \end{itemize}
  Then there exists a constant $c_T>0$ depending on $T$ but independent of $\varepsilon$ such that
  \begin{align} \label{E:Error_Gen}
    \begin{aligned}
      \|\rho^\varepsilon-\bar{\eta}\|_{C(\overline{Q_{\varepsilon,T}})} &\leq c_T\Bigl(\|\rho_0^\varepsilon-\bar{\eta}_0\|_{C(\overline{\Omega_\varepsilon(0)})}+\|f^\varepsilon-\bar{f}\|_{C(Q_{\varepsilon,T})}\Bigr) \\
      &\qquad +\varepsilon c_T\Bigl(\|\eta_0\|_{C(\Gamma(0))}+\|\eta\|_{C^{2,1}(S_T)}\Bigr) \\
      &\qquad +\varepsilon c_T\sum_{i=0,1}\Bigl(\|h_i(\cdot,0)\|_{C(\Gamma(0))}+\|h_i\|_{C^{2,1}(S_T)}\Bigr)
    \end{aligned}
  \end{align}
  provided that the right-hand side is finite.
\end{theorem}

The condition (b) is satisfied if $\eta_0$ and $f$ are sufficiently smooth by a parabolic regularity theory.
Also, the last line of \eqref{E:Error_Gen} can be estimated by the H\"{o}lder norms of $\eta_0$ and $f$, but higher order norms will be required compared to the $C^{2,1}(S_T)$-estimate of $\eta$.

The proofs of Theorem \ref{T:Error_Spe} and \ref{T:Error_Gen} are given in Section \ref{S:Pf_Main}.
We explain the idea of the proofs in the next subsection.

In \eqref{E:Error_Spe} and \eqref{E:Error_Gen}, the volume of $Q_{\varepsilon,T}$ and $\Omega_\varepsilon(0)$ is irrelevant to the sup-norm.
Thus, by Theorems \ref{T:Error_Spe} and \ref{T:Error_Gen}, we can say that a classical solution to \eqref{E:Limit} approximates a classical solution to \eqref{E:Heat_MTD} of order $\varepsilon$ in the sup-norm as in the $L^2$-framework.
We should note that, however, unlike Theorem \ref{T:Error_Gen}, an additional regularity assumption on an $L^2$-weak solution to \eqref{E:Limit} is not required for the $L^2$-error estimate \eqref{E:Error_L2} even if $\nabla_\Gamma g_i\not\equiv0$ for $i=0,1$ (see \cite{Miu17}).
This is due to the fact that we use the strong form of \eqref{E:Heat_MTD} to show \eqref{E:Error_Gen}, while the weak form of \eqref{E:Heat_MTD} is used in the proof of \eqref{E:Error_L2}.
As we explain below, we need to construct a suitable approximate solution to the strong form of \eqref{E:Heat_MTD} involving $\nabla_\Gamma g_i\cdot\nabla_\Gamma\eta$, $i=0,1$ from a classical solution $\eta$ to \eqref{E:Limit} in the proof of \eqref{E:Error_Gen}, since $\eta$ itself does not satisfy the strong form of \eqref{E:Heat_MTD} even approximately.
Hence the condition (b) of Theorem \ref{T:Error_Gen} is required in order to make the approximate solution well-defined at $t=0$ and sufficiently regular in $Q_{\varepsilon,T}$.
In the proof of \eqref{E:Error_L2}, however, we take the average in the thin direction of the weak form of \eqref{E:Heat_MTD} and use the average of a weak solution to \eqref{E:Heat_MTD} as an approximate weak solution to \eqref{E:Limit}.
This method generates the $L^2(Q_{\varepsilon,T})$-norms of a weak solution to \eqref{E:Heat_MTD} and its gradient as additional terms in an $L^2$-error estimate, but they can be bounded by the $L^2(\Omega_\varepsilon(0))$-norm of an initial data (and an appropriate norm of a source term) by a standard energy estimate.
Hence we can avoid additional regularity assumptions on weak solutions to \eqref{E:Heat_MTD} and \eqref{E:Limit} even if $\nabla_\Gamma g_i\not\equiv0$ for $i=0,1$.

By the above explanations, one may guess that the condition (b) of Theorem \ref{T:Error_Gen} can be removed by the use of the average method.
It may be possible to do that, but the average method will generate the $C(Q_{\varepsilon,T})$-norms of the derivatives of a classical solution $\rho^\varepsilon$ to \eqref{E:Heat_MTD} up to the second order (or more) as additional terms in an error estimate instead of the last line of \eqref{E:Error_Gen}.
Then, unlike the norms in the last line of \eqref{E:Error_Gen}, the $C(Q_{\varepsilon,T})$-norms of the derivatives of $\rho^\varepsilon$ depend on $\varepsilon$, and we do not a priori know how these norms depend on $\varepsilon$.
Hence we need to estimate the derivatives of $\rho^\varepsilon$ in $C(Q_{\varepsilon,T})$ explicitly in terms of $\varepsilon$ by the given data $\rho_0^\varepsilon$ and $f^\varepsilon$, but such a task will be very tough and also require an additional regularity assumption on $\rho_0^\varepsilon$ and $f^\varepsilon$ (see e.g. \cite{LaSoUr68,Lie96} in the case of a general domain).
We would like to avoid such a tough work and an assumption on $\rho_0^\varepsilon$ and $f^\varepsilon$, so we do not use the average method in this paper.

\subsection{Idea of the proofs} \label{SS:In_Idea}
The proofs of Theorems \ref{T:Error_Spe} and \ref{T:Error_Gen} are based on a uniform a priori estimate for a classical solution to \eqref{E:Heat_MTD} and a construction of a suitable approximate solution to \eqref{E:Heat_MTD} from a classical solution to \eqref{E:Limit}.
For classical solutions $\rho^\varepsilon$ and $\eta$ to \eqref{E:Heat_MTD} and \eqref{E:Limit}, we give an approximate solution $\rho_\eta^\varepsilon$ to \eqref{E:Heat_MTD} close to $\eta$ so that the difference of $\rho^\varepsilon$ and $\rho_\eta^\varepsilon$ satisfies
\begin{align} \label{E:Int_DfEq}
  \left\{
  \begin{alignedat}{3}
    \partial_t(\rho^\varepsilon-\rho_\eta^\varepsilon)-k_d\Delta(\rho^\varepsilon-\rho_\eta^\varepsilon) &= (f^\varepsilon-\bar{f})-f_\eta^\varepsilon &\quad &\text{in} &\quad &Q_{\varepsilon,T}, \\
    \partial_{\nu_\varepsilon}(\rho^\varepsilon-\rho_\eta^\varepsilon)+k_d^{-1}V_\varepsilon(\rho^\varepsilon-\rho_\eta^\varepsilon) &= -\psi_\eta^\varepsilon &\quad &\text{on} &\quad &\partial_\ell Q_{\varepsilon,T}, \\
    (\rho^\varepsilon-\rho_\eta^\varepsilon)|_{t=0} &= \rho_0^\varepsilon-\rho_\eta^\varepsilon(\cdot,0) &\quad &\text{in} &\quad &\Omega_\varepsilon(0),
  \end{alignedat}
  \right.
\end{align}
where $f_\eta^\varepsilon$ and $\psi_\eta^\varepsilon$ are error terms due to $\rho_\eta^\varepsilon$.
Then we prove the uniform a priori estimate
\begin{align} \label{E:Int_UAP}
  \begin{aligned}
    \|\rho^\varepsilon-\rho_\eta^\varepsilon\|_{C(\overline{Q_{\varepsilon,T}})} &\leq c\Bigl(\|\rho_0^\varepsilon-\rho_\eta^\varepsilon(\cdot,0)\|_{C(\overline{\Omega_\varepsilon(0)})}+\|f^\varepsilon-\bar{f}\|_{C(Q_{\varepsilon,T})}\Bigr) \\
    &\qquad +c\Bigl(\|f_\eta^\varepsilon\|_{C(Q_{\varepsilon,T})}+\frac{1}{\varepsilon}\|\psi_\eta^\varepsilon\|_{C(\partial_\ell Q_{\varepsilon,T})}\Bigr)
  \end{aligned}
\end{align}
for the classical solution $\rho^\varepsilon-\rho_\eta^\varepsilon$ to \eqref{E:Int_DfEq}, where $c>0$ is a constant independent of $\varepsilon$ (see Theorem \ref{T:Uni_Heat}).
Hence we can get \eqref{E:Error_Spe} and \eqref{E:Error_Gen} by choosing $\rho_\eta^\varepsilon$ so that $\rho_\eta^\varepsilon-\bar{\eta}$ and $f_\eta^\varepsilon$ are of order $\varepsilon$ in $\overline{Q_{\varepsilon,T}}$ and $Q_{\varepsilon,T}$ and $\psi_\eta^\varepsilon$ is of order $\varepsilon^2$ on $\partial_\ell Q_{\varepsilon,T}$.

To prove \eqref{E:Int_UAP}, we consider a general parabolic equation
\begin{align} \label{E:Int_Para}
  \left\{
  \begin{alignedat}{3}
    \partial_t\chi^\varepsilon-\sum_{i,j=1}^na_{ij}^\varepsilon\partial_i\partial_j\chi^\varepsilon+\sum_{i=1}^nb_i^\varepsilon\partial_i\chi^\varepsilon+c^\varepsilon\chi^\varepsilon &= f^\varepsilon &\quad &\text{in} &\quad &Q_{\varepsilon,T}, \\
    \partial_{\nu_\varepsilon}\chi^\varepsilon+\beta^\varepsilon\chi^\varepsilon &= \psi^\varepsilon &\quad &\text{on} &\quad &\partial_\ell Q_{\varepsilon,T}, \\
    \chi^\varepsilon|_{t=0} &= \chi_0^\varepsilon &\quad &\text{in} &\quad &\Omega_\varepsilon(0).
  \end{alignedat}
  \right.
\end{align}
By a standard argument based on the maximum principle (see e.g. \cite{LaSoUr68,Lie96}), we show that a uniform a priori estimate similar to \eqref{E:Int_UAP} holds for a classical solution $\chi^\varepsilon$ to \eqref{E:Int_Para} provided that $a_{ij}^\varepsilon$, $b_i^\varepsilon$, and $c^\varepsilon$ are uniformly bounded in $Q_{\varepsilon,T}$ and $\beta^\varepsilon\geq-c\varepsilon$ on $\partial_\ell Q_{\varepsilon,T}$ with some constant $c>0$ independent of $\varepsilon$ (see Theorem \ref{T:Uni_Para}).
Here we note that the condition on $\beta^\varepsilon$ cannot be replaced by the uniform lower bound $\beta^\varepsilon\geq-c$ due to the thickness of $\Omega_\varepsilon(t)$ and the fact that the direction of $\nu_\varepsilon$ on the inner boundary of $\Omega_\varepsilon(t)$ is opposite to that of $\nu_\varepsilon$ on the outer boundary of $\Omega_\varepsilon(t)$ (see Remark \ref{R:Uni_Para}).
Moreover, the factor $\varepsilon^{-1}$ appears in \eqref{E:Int_UAP} because of the condition on $\beta^\varepsilon$, and we cannot remove it since we have a contradiction if \eqref{E:Int_UAP} holds without $\varepsilon^{-1}$ (see Remark \ref{R:PM_UapE}).
In order to apply the result of Theorem \ref{T:Uni_Para} to \eqref{E:Int_DfEq}, we also need to deal with the outer normal velocity $V_\varepsilon$ of $\partial\Omega_\varepsilon(t)$.
When $\varepsilon$ is sufficiently small, $V_\varepsilon$ is close to $-V_\Gamma$ on the inner boundary of $\Omega_\varepsilon(t)$ and to $V_\Gamma$ on the outer boundary of $\Omega_\varepsilon(t)$, where $V_\Gamma$ is the outer normal velocity of $\Gamma(t)$ (see Lemma \ref{L:MTD_ONV}).
Hence $V_\varepsilon$ is of order one with respect to $\varepsilon$ and can take both the positive and negative values on the whole boundary $\partial\Omega_\varepsilon(t)$, and thus $V_\varepsilon$ does not satisfy the condition on $\beta^\varepsilon$ in general.
To overcome this difficulty, we eliminate the zeroth order term $\pm V_\Gamma$ of $V_\varepsilon$ by multiplying $\rho^\varepsilon-\rho_\eta^\varepsilon$ by the exponential of a suitable function involving $V_\Gamma$.
Let us also mention that we do not scale the thickness of $\Omega_\varepsilon(t)$ and use local coordinates of $\Gamma(t)$ in the proofs of the uniform a priori estimates.
Our proofs avoid complicated calculations associated with the change of variables for the curved thin domain $\Omega_\varepsilon(t)$ with complicated geometry, so they seem to be readable and easy to understand.

Another task is to find a suitable approximate solution $\rho_\eta^\varepsilon$ to \eqref{E:Heat_MTD} close to $\eta$.
In order to derive \eqref{E:Error_Spe} and \eqref{E:Error_Gen} from \eqref{E:Int_UAP}, we need to choose $\rho_\eta^\varepsilon$ so that the error terms $f_\eta^\varepsilon$ and $\psi_\eta^\varepsilon$ are of order $\varepsilon$ and $\varepsilon^2$, respectively, i.e. $\rho_\eta^\varepsilon$ should satisfy \eqref{E:Heat_MTD} approximately of order $\varepsilon$ in $Q_{\varepsilon,T}$ and $\varepsilon^2$ on $\partial_\ell Q_{\varepsilon,T}$.
It turns out that $\eta$, the classical solution to \eqref{E:Limit}, itself does not satisfy \eqref{E:Heat_MTD} approximately due to the geometry and motion of $\partial\Omega_\varepsilon(t)$.
Hence we seek a suitable approximate solution by taking an asymptotic expansion
\begin{align*}
  \rho^\varepsilon(x,t) &= \sum_{k=0}^\infty\varepsilon^k\eta_k(\pi(x,t),t,\varepsilon^{-1}d(x,t)), \quad (x,t)\in\overline{Q_{\varepsilon,T}},
\end{align*}
where $\pi(x,t)$ is the closest point of $x$ on $\Gamma(t)$, $d(x,t)$ is the signed distance of $x$ from $\Gamma(t)$, and the functions
\begin{align*}
  \eta_k(y,t,r), \quad (y,t)\in\overline{S_T}, \, r\in[g_0(y,t),g_1(y,t)], \, k=0,1,\dots
\end{align*}
are independent of $\varepsilon$.
In Section \ref{S:Formal}, we substitute the right-hand side of the above expansion for \eqref{E:Heat_MTD} and determine $\eta_k$ for which \eqref{E:Heat_MTD} is satisfied of order $\varepsilon$ in $Q_{\varepsilon,T}$ and $\varepsilon^2$ on $\partial_\ell Q_{\varepsilon,T}$.
Then we find that $\eta_0=\eta$ should be a solution to \eqref{E:Limit} and it is enough to determine $\eta_1$ and $\eta_2$ for our purpose, but $\eta_2$ involves $\nabla_\Gamma g_i\cdot\nabla_\Gamma\eta_0$, $i=0,1$ in order to make the boundary condition of \eqref{E:Heat_MTD} satisfied of order $\varepsilon^2$.
This is the reason why we require the additional regularity assumption (b) on $\eta$ in Theorem \ref{T:Error_Gen}.
We also note that the limit equation \eqref{E:Limit} appears as a necessary condition on $\eta_0$ for the existence of $\eta_2$.

In the actual proofs of Theorems \ref{T:Error_Spe} and \ref{T:Error_Gen} given in Section \ref{S:Pf_Main}, we use the functions $\eta_1$ and $\eta_2$ to define the approximate solution $\rho_\eta^\varepsilon$ by
\begin{align*}
  \rho_\eta^\varepsilon(x,t) = \bar{\eta}(x,t)+\sum_{k=1,2}\varepsilon^k\eta_k(\pi(x,t),t,\varepsilon^{-1}d(x,t)), \quad (x,t)\in\overline{Q_{\varepsilon,T}}.
\end{align*}
This $\rho_\eta^\varepsilon$ is expressed as the sum of functions of the form
\begin{align*}
  \Bigl(d^k\bar{\zeta}\Bigr)(x,t) = d(x,t)^k\bar{\zeta}(x,t), \quad (x,t)\in\overline{Q_{\varepsilon,T}}, \, k=0,1,2,
\end{align*}
where $\zeta$ is a function on $\overline{S_T}$ and $\bar{\zeta}$ is its constant extension in the normal direction of $\Gamma(t)$.
Then we express the derivatives of $d^k\bar{\zeta}$ approximately in terms of functions on $\overline{S_T}$ by using lemmas in Section \ref{S:Prelim}, and apply the resulting expressions, the explicit forms of $\eta_1$ and $\eta_2$, and the fact that $\eta$ satisfies \eqref{E:Limit} to show that $\rho_\eta^\varepsilon$ indeed satisfies \eqref{E:Heat_MTD} approximately of order $\varepsilon$ in $Q_{\varepsilon,T}$ and $\varepsilon^2$ on $\partial_\ell Q_{\varepsilon,T}$.
Here we again note that the proofs of Theorems \ref{T:Error_Spe} and \ref{T:Error_Gen} avoid the scaling of the thickness of $\Omega_\varepsilon(t)$ and the use of local coordinates of $\Gamma(t)$.

\subsection{Literature overview} \label{SS:In_Lite}
Partial differential equations (PDEs) in thin domains appear in various fields like engineering, biology, and fluid mechanics.
Many authors have studied PDEs (especially reaction-diffusion and the Navier--Stokes equations) in flat thin domains around a lower dimensional domain since the pioneering works by Hale and Raugel \cite{HalRau92_DH,HalRau92_RD}.
Moreover, there are several works on a reaction-diffusion equation and its stationary problem in a thin L-shaped domain \cite{HalRau95}, in a flat thin domain with holes \cite{PriRyb01}, in thin tubes around curves or networks \cite{Yan90,Kos00,Kos02}, and in curved thin domains around lower dimensional manifolds \cite{PrRiRy02,PriRyb03}.
Curved thin domains around (hyper)surfaces also appears in the study of the Navier--Stokes equations \cite{TemZia97,Miu20_03,Miu21_02,Miu22_01} and of the asymptotic behavior of the eigenvalues of the Laplacian \cite{Sch96,Kre14,JimKur16,Yac18}.
We also refer to \cite{Rau95} for various examples of thin domains.

In the above cited papers, thin domains and their limit sets are stationary in time.
Pereira and Silva \cite{PerSil13} first studied a reaction-diffusion equation in a moving thin domain which has a moving boundary but whose limit set is a stationary domain.
Elliott and Stinner \cite{EllSti09} considered a moving thin domain around a moving surface in order to approximate a given advection-diffusion equation on a moving surface by a diffuse interface model (see also \cite{ElStStWe11} for a numerical computation of the diffuse interface model).
In \cite{Miu17}, the present author first rigorously derived an unknown limit equation on a moving surface from a given equation in a moving thin domain by the average method in the case of the heat equation.
Also, fluid and nonlinear diffusion equations on moving surfaces are formally derived by the thin-film limit in \cite{Miu18,MiGiLi18}, but the justification has not been done yet because of difficulties coming from the nonlinearity of the equations.

PDEs on moving surfaces also have attracted interest of many researchers recently in view of applications.
There are many works on the mathematical and numerical analysis of linear diffusion equations on moving surfaces like \eqref{E:Limit} (see e.g. \cite{DziEll07,DziEll13_AN,DziEll13_MC,OlReXu14,Vie14,AlElSt15_IFB,EllFri15} and the references cited therein).
Other equations on moving surfaces were also studied such as a Stefan problem \cite{AlpEll15}, a porous medium equation \cite{AlpEll16}, the Cahn--Hilliard equation \cite{EllRan15,CaeEll21,OCoSti16}, and the Hamilton--Jacobi equation \cite{DeElMiSt19}.
The authors of \cite{YaOzSa16} formulated equations of nonlinear elasticity in an evolving ambient space like a moving surface.
Also, the Stokes and Navier--Stokes equations on moving surfaces were proposed in \cite{ArrDeS09} and \cite{KoLiGi17,JaOlRe18}.
These fluid equations are coupled systems of the motion of a surface described by the normal velocity and the evolution of a tangential fluid velocity.
In \cite{SaYaSaMa20,OlReZh22}, numerical methods for the Navier--Stokes equations on a moving surface were proposed.
Also, when the motion of a surface is given, the wellposedness of the tangential Navier--Stokes equations on a moving surface was shown in \cite{OlReZh22}.
Some models of liquid crystals on moving surfaces were proposed and numerically computed in \cite{NiReVo19,NiReVo20}.
In \cite{DecSty22,ElGaKo22,AbBuGa22_QP,AbBuGa22_ST}, a coupled system of a mean curvature flow for a surface and a diffusion equation on the surface was studied.
We also refer to \cite{AlElSt15_PM,AlCaDjEl21} for abstract frameworks for PDEs in evolving function spaces.

\subsection{Organization of this paper} \label{SS:In_Orga}
The rest of this paper is organized as follows.
We fix notations and give basic results on surfaces and thin domains in Section \ref{S:Prelim}.
In Section \ref{S:Bulk} we show a uniform a priori estimate for a classical solution to the heat equation in the moving thin domain.
Section \ref{S:Pf_Main} is devoted to the proofs of Theorems \ref{T:Error_Spe} and \ref{T:Error_Gen}.
In Section \ref{S:Formal} we explain a formal derivation of the limit equation \eqref{E:Limit} and a suitable approximate solution used in the proofs of Theorems \ref{T:Error_Spe} and \ref{T:Error_Gen} based on an asymptotic expansion of the thin domain problem \eqref{E:Heat_MTD}.

\section{Preliminaries} \label{S:Prelim}
In this section we fix notations and give basic results on surfaces and thin domains.
We assume that functions and surfaces appearing below are sufficiently smooth.
Also, we consider a vector in $\mathbb{R}^n$, $n\geq2$ as a column vector, and for a square matrix $A$ we denote by $A^T$ the transpose of $A$ and by $|A| = \sqrt{\mathrm{tr}[A^TA]}$ the Frobenius norm of $A$.

\subsection{Fixed surface} \label{SS:FixSurf}
Let $\Gamma$ be a closed (i.e. compact and without boundary), connected, and oriented smooth hypersurface in $\mathbb{R}^n$, $n\geq2$.
We denote by $\nu$ and $d$ the unit outward normal vector field of $\Gamma$ and the signed distance function from $\Gamma$ increasing in the direction of $\nu$.
Also, we write $\kappa_1,\dots,\kappa_{n-1}$ for the principal curvatures of $\Gamma$.
Then $\nu$ and $\kappa_1,\dots,\kappa_{n-1}$ are smooth and thus bounded on $\Gamma$ by the smoothness of $\Gamma$.
Hence we may take a tubular neighborhood $\overline{N}=\{x\in\mathbb{R}^n \mid -\delta\leq d(x)\leq \delta\}$ of $\Gamma$ with $\delta>0$ such that for each $x\in\overline{N}$ there exists a unique $\pi(x)\in\Gamma$ satisfying
\begin{align} \label{E:Fermi}
  x = \pi(x)+d(x)\nu(\pi(x)), \quad \nabla d(x) = \nabla d(\pi(x)) = \nu(\pi(x))
\end{align}
and that $d$ and $\pi$ are smooth on $\overline{N}$ (see e.g. \cite[Section 14.6]{GilTru01}).
Moreover, taking $\delta>0$ sufficiently small, we may assume that
\begin{align} \label{E:Pr_Curv}
  c_0^{-1} \leq 1-r\kappa_\alpha(y) \leq c_0, \quad y\in\Gamma, \, r\in[-\delta,\delta], \, \alpha=1,\dots,n-1
\end{align}
with some constant $c_0>0$.

Let $I_n$ and $\nu\otimes\nu$ be the $n\times n$ identity matrix and the tensor product of $\nu$ with itself.
We set $P=(P_{ij})_{i,j}=I_n-\nu\otimes\nu$ on $\Gamma$, which is the orthogonal projection onto the tangent plane of $\Gamma$.
For a function $\eta$ on $\Gamma$, we define the tangential gradient and the tangential derivatives of $\eta$ by
\begin{align*}
  \nabla_\Gamma\eta(y) = P(y)\nabla\tilde{\eta}(y), \quad \underline{D}_i\eta(y) = \sum_{i=1}^nP_{ij}(y)\partial_j\tilde{\eta}(y), \quad y\in\Gamma, \, i=1,\dots,n
\end{align*}
so that $\nabla_\Gamma\eta=(\underline{D}_1\eta,\dots,\underline{D}_n\eta)^T$, where $\tilde{\eta}$ is an extension of $\eta$ to $\overline{N}$.
Then
\begin{align} \label{E:TGr_Tang}
  \nu\cdot\nabla_\Gamma\eta = 0, \quad P\nabla_\Gamma\eta = \nabla_\Gamma\eta \quad\text{on}\quad \Gamma
\end{align}
and the values of $\nabla_\Gamma\eta$ and $\underline{D}_i\eta$ are independent of the choice of $\tilde{\eta}$.
In particular, if we take the constant extension $\bar{\eta}=\eta\circ\pi$ of $\eta$ in the normal direction of $\Gamma$, then we have
\begin{align} \label{E:CEGr_Surf}
  \nabla\bar{\eta}(y) = \nabla_\Gamma\eta(y), \quad \partial_i\bar{\eta}(y) = \underline{D}_i\eta(y), \quad y\in\Gamma, \, i=1,\dots,n
\end{align}
by \eqref{E:Fermi} and $d(y)=0$ for $\Gamma$.
In what follows, a function with an overline always stands for the constant extension of a function on $\Gamma$ in the normal direction of $\Gamma$.
We set
\begin{align*}
  \Delta_\Gamma\eta = \sum_{i=1}^n\underline{D}_i\underline{D}_i\eta, \quad |\nabla_\Gamma^2\eta| = \left(\sum_{i,j=1}^n|\underline{D}_i\underline{D}_j\eta|^2\right)^{1/2} \quad\text{on}\quad \Gamma
\end{align*}
and call $\Delta_\Gamma$ the Laplace--Beltrami operator on $\Gamma$.
For a (not necessarily tangential) vector field $u$ on $\Gamma$, we define the surface divergence of $u$ by
\begin{align*}
  \mathrm{div}_\Gamma u = \sum_{i=1}^n\underline{D}_iu_i \quad\text{on}\quad \Gamma,
\end{align*}
where $u=(u_1,\dots,u_n)^T$.
Also, we define
\begin{align*}
  W_{ij} = -\underline{D}_i\nu_j, \quad H = -\mathrm{div}_\Gamma\nu \quad\text{on}\quad \Gamma, \quad i,j=1,\dots,n
\end{align*}
and call $W=(W_{ij})_{i,j}$ and $H$ the Weingarten map (or the shape operator) and the mean curvature of $\Gamma$.
The matrix $W$ is symmetric since $W=-\nabla^2d$ on $\Gamma$ by \eqref{E:Fermi} and \eqref{E:CEGr_Surf}.
Moreover, $W\nu=-\nabla_\Gamma(|\nu|^2/2)=0$ by $|\nu|=1$ on $\Gamma$ and thus $W$ has the eigenvalue zero and $WP=PW=W$ on $\Gamma$.
It is also known (see e.g. \cite{Lee18}) that the other eigenvalues of $W$ are the principal curvatures $\kappa_1,\dots,\kappa_{n-1}$.
Hence
\begin{align*}
  H = \mathrm{tr}[W] = \sum_{\alpha=1}^{n-1}\kappa_\alpha, \quad |W|^2 = \mathrm{tr}[W^2] = \sum_{\alpha=1}^{n-1}\kappa_\alpha^2 \quad\text{on}\quad \Gamma.
\end{align*}
In particular, we have
\begin{align} \label{E:HW_Bound}
  \begin{aligned}
    |H(y)| &\leq (n-1)\max_{\alpha=1,\dots,n-1}\max_{z\in\Gamma}|\kappa_\alpha(z)|, \\
    |W(y)| &\leq (n-1)^{1/2}\max_{\alpha=1,\dots,n-1}\max_{z\in\Gamma}|\kappa_\alpha(z)|
  \end{aligned}
\end{align}
for all $y\in\Gamma$.
Also, it follows from \eqref{E:Pr_Curv} that the matrix
\begin{align*}
  I_n-d(x)\overline{W}(x) = I_n-rW(y)
\end{align*}
is invertible for all $x=y+rn(y)\in\overline{N}$ with $y\in\Gamma$ and $r\in[-\delta,\delta]$.
We set
\begin{align} \label{E:Inv_IdW}
  R(x) = \bigl(R_{ij}(x)\bigr)_{i,j} = \Bigl\{I_n-d(x)\overline{W}(x)\Bigr\}^{-1}, \quad x\in\overline{N}.
\end{align}

Let us give a few lemmas related to the derivatives of the constant extension of a function on $\Gamma$.
In Lemmas \ref{L:R_Norm}--\ref{L:CE_Lap} below, we denote by $c$ a general positive constant depending only on $n$, $\delta$, the constant $c_0$ appearing in \eqref{E:Pr_Curv}, and the quantities
\begin{align*}
  \max_{\alpha=1,\dots,n-1}\max_{y\in\Gamma}|\kappa_\alpha(y)|, \quad \max_{i,j,k=1,\dots,n}\max_{y\in\Gamma}|\underline{D}_iW_{jk}(y)|.
\end{align*}

\begin{lemma} \label{L:R_Norm}
  For all $x\in\overline{N}$ we have
  \begin{align} \label{E:R_Norm}
    |R(x)| \leq c, \quad |I_n-R(x)| \leq c|d(x)|.
  \end{align}
\end{lemma}

\begin{proof}
  For $x\in\overline{N}$ let $y=\pi(x)\in\Gamma$ and $r=d(x)\in[-\delta,\delta]$ so that
  \begin{align*}
    R(x) = \{I_n-rW(y)\}^{-1}.
  \end{align*}
  Since $W(y)$ is symmetric and has the eigenvalues $\kappa_1(y),\dots,\kappa_{n-1}(y)$ and zero, we can take an orthonormal basis of $\mathbb{R}^n$ consisting of the corresponding eigenvectors of $W(y)$.
  Using this, we easily find that
  \begin{align*}
    |R(x)|^2 = \sum_{\alpha=1}^{n-1}\{1-r\kappa_\alpha(y)\}^{-2}+1, \quad |I_n-R(x)|^2 = \sum_{\alpha=1}^{n-1}\left(\frac{r\kappa_\alpha(y)}{1-r\kappa_\alpha(y)}\right)^2.
  \end{align*}
  Hence we obtain \eqref{E:R_Norm} by \eqref{E:Pr_Curv} and $r=d(x)$.
\end{proof}

\begin{lemma} \label{L:CEGr_TN}
  Let $\eta$ be a function on $\Gamma$ and $\bar{\eta}=\eta\circ\pi$ be its constant extension.
  Then
  \begin{align} \label{E:CEGr_TN}
    \nabla\bar{\eta}(x) = R(x)\overline{\nabla_\Gamma\eta}(x), \quad \bar{\nu}(x)\cdot\nabla\bar{\eta}(x) = 0
  \end{align}
  for all $x\in\overline{N}$.
  Moreover,
  \begin{align}
    |\nabla\bar{\eta}(x)| &\leq c\Bigl|\overline{\nabla_\Gamma\eta}(x)\Bigr|, \label{E:CEGr_Bound} \\
    \Bigl|\nabla\bar{\eta}(x)-\overline{\nabla_\Gamma\eta}(x)\Bigr| &\leq c|d(x)|\Bigl|\overline{\nabla_\Gamma\eta}(x)\Bigr|. \label{E:CEGr_Diff}
  \end{align}
\end{lemma}

\begin{proof}
  For the mapping $\pi=(\pi_1,\dots,\pi_n)$ given in \eqref{E:Fermi}, we write
  \begin{align*}
    \nabla\pi =
    \begin{pmatrix}
      \partial_1\pi_1 & \cdots & \partial_1\pi_n \\
      \vdots & \ddots & \vdots \\
      \partial_n\pi_1 & \cdots & \partial_n\pi_n
    \end{pmatrix}.
  \end{align*}
  Since $\pi(x)=x-d(x)\bar{\nu}(\pi(x))$ by \eqref{E:Fermi}, we see by \eqref{E:CEGr_Surf} with $y=\pi(x)$ that
  \begin{align*}
    \nabla\pi(x)\Bigl\{I_n-d(x)\overline{W}(x)\Bigr\} = \overline{P}(x), \quad \nabla\pi(x) = \overline{P}(x)R(x) = R(x)\overline{P}(x),
  \end{align*}
  where the last equality is due to $PW=WP$ on $\Gamma$.
  We differentiate $\bar{\eta}(x)=\bar{\eta}(\pi(x))$ and use the above equality, \eqref{E:TGr_Tang}, and \eqref{E:CEGr_Surf} to get the first equality of \eqref{E:CEGr_TN}.
  Also, the second equality holds since $\bar{\eta}$ is the constant extension of $\eta$ in the direction of $\bar{\nu}$.
  The inequalities \eqref{E:CEGr_Bound} and \eqref{E:CEGr_Diff} follow from \eqref{E:R_Norm} and \eqref{E:CEGr_TN}.
\end{proof}

\begin{lemma} \label{L:CEGr_2nd}
  For all $x\in\overline{N}$ and $i=1,\dots,n$ we have
  \begin{align} \label{E:IdW_Der}
    \Bigl|\partial_iR(x)-\bar{\nu}_i(x)\overline{W}(x)\Bigr| \leq c|d(x)|.
  \end{align}
  Also, let $\eta$ be a function on $\Gamma$ and $\bar{\eta}=\eta\circ\pi$ be its constant extension.
  Then
  \begin{align} \label{E:CG2_Bo}
    |\nabla^2\bar{\eta}(x)| \leq c\left(\Bigl|\overline{\nabla_\Gamma\eta}(x)\Bigr|+\Bigl|\overline{\nabla_\Gamma^2\eta}(x)\Bigr|\right)
  \end{align}
  for all $x\in\overline{N}$, where $\nabla^2\bar{\eta}=(\partial_i\partial_j\bar{\eta})_{i,j}$, and
  \begin{multline} \label{E:CG2_Diff}
    \left|\partial_i\partial_j\bar{\eta}(x)-\overline{\underline{D}_i\underline{D}_j\eta}(x)-\bar{\nu}_i(x)\sum_{k=1}^n\overline{W}_{jk}(x)\overline{\underline{D}_k\eta}(x)\right| \\
    \leq c|d(x)|\left(\Bigl|\overline{\nabla_\Gamma\eta}(x)\Bigr|+\Bigl|\overline{\nabla_\Gamma^2\eta}(x)\Bigr|\right)
  \end{multline}
  for all $x\in\overline{N}$ and $i,j=1,\dots,n$.
\end{lemma}

\begin{proof}
  We apply $\partial_i$ to
  \begin{align*}
    R(x)\Bigl\{I_n-d(x)\overline{W}(x)\Bigr\} = I_n, \quad x\in\overline{N}
  \end{align*}
  and use $\nabla d=\bar{\nu}$ in $\overline{N}$ to find that
  \begin{align*}
    \partial_iR(x) = \bar{\nu}_i(x)R(x)\overline{W}(x)R(x)+d(x)R(x)\partial_i\overline{W}(x)R(x).
  \end{align*}
  By this equality, \eqref{E:HW_Bound}, \eqref{E:R_Norm}, \eqref{E:CEGr_Bound}, and the boundedness of $\underline{D}_iW_{jk}$ on $\Gamma$, we obtain \eqref{E:IdW_Der}.
  Also, for a function $\eta$ on $\Gamma$, it follows from \eqref{E:CEGr_TN} that
  \begin{align*}
    \partial_j\bar{\eta}(x) = \sum_{k=1}^nR_{jk}(x)\overline{\underline{D}_k\eta}(x), \quad x\in\overline{N}.
  \end{align*}
  We apply $\partial_i$ to both sides and use \eqref{E:CEGr_TN} to $\underline{D}_k\eta$ to get
  \begin{align*}
    \partial_i\partial_j\bar{\eta}(x) = \sum_{k=1}^n\partial_iR_{jk}(x)\overline{\underline{D}_k\eta}(x)+\sum_{k,l=1}^nR_{jk}(x)R_{il}(x)\overline{\underline{D}_l\underline{D}_k\eta}(x).
  \end{align*}
  Hence we have \eqref{E:CG2_Bo} and \eqref{E:CG2_Diff} by this equality, \eqref{E:R_Norm}, and \eqref{E:IdW_Der}.
\end{proof}

The next lemma plays a fundamental role in the proofs of Theorems \ref{T:Error_Spe} and \ref{T:Error_Gen}.

\begin{lemma} \label{L:CE_Lap}
  Let $\eta$ be a function on $\Gamma$ and $\bar{\eta}=\eta\circ\pi$ be its constant extension.
  Then
  \begin{align}
    \Bigl|\Delta\bar{\eta}(x)-\overline{\Delta_\Gamma\eta}(x)\Bigr| &\leq c|d(x)|\left(\Bigl|\overline{\nabla_\Gamma\eta}(x)\Bigr|+\Bigl|\overline{\nabla_\Gamma^2\eta}(x)\Bigr|\right), \label{E:CE_Lap} \\
    \Bigl|\Delta(d\bar{\eta})(x)+\overline{H}(x)\bar{\eta}(x)\Bigr| &\leq c|d(x)|\left(|\bar{\eta}(x)|+\Bigl|\overline{\nabla_\Gamma\eta}(x)\Bigr|+\Bigl|\overline{\nabla_\Gamma^2\eta}(x)\Bigr|\right), \label{E:CEL_1d} \\
    \Bigl|\Delta(d^2\bar{\eta})(x)-2\bar{\eta}(x)\Bigr| &\leq c|d(x)|\left(|\bar{\eta}(x)|+\Bigl|\overline{\nabla_\Gamma\eta}(x)\Bigr|+\Bigl|\overline{\nabla_\Gamma^2\eta}(x)\Bigr|\right) \label{E:CEL_2d}
  \end{align}
  for all $x\in\overline{N}$.
  Here we write $(d^k\bar{\eta})(x)=d(x)^k\bar{\eta}(x)$ for $x\in\overline{N}$ and $k=1,2$.
\end{lemma}

\begin{proof}
  The inequality \eqref{E:CE_Lap} follows from \eqref{E:CG2_Diff} and
  \begin{align*}
    \sum_{i,k=1}^n\nu_iW_{ik}\underline{D}_k\eta = (W\nu)\cdot\nabla_\Gamma\eta = 0 \quad\text{on}\quad \Gamma
  \end{align*}
  by the symmetry of $W$ and $W\nu=0$ on $\Gamma$.
  Next we see that
  \begin{align*}
    \Delta(d\bar{\eta}) = (\mathrm{div}\,\bar{\nu})\bar{\eta}+2\bar{\nu}\cdot\nabla\bar{\eta}+d\Delta\bar{\eta} = (\mathrm{div}\,\bar{\nu})\bar{\eta}+d\Delta\bar{\eta} \quad\text{in}\quad \overline{N}
  \end{align*}
  by $\nabla d=\bar{\nu}$ in $\overline{N}$ and \eqref{E:CEGr_TN}.
  Moreover,
  \begin{align} \label{Pf_CEL:div_nu}
    \Bigl|\mathrm{div}\,\bar{\nu}+\overline{H}\Bigr| = \Bigl|\mathrm{div}\,\bar{\nu}-\overline{\mathrm{div}_\Gamma\nu}\Bigr| \leq c|d|\Bigl|\overline{W}\Bigr| \leq c|d| \quad\text{in}\quad \overline{N}
  \end{align}
  by $H=-\mathrm{div}_\Gamma\nu$ on $\Gamma$, \eqref{E:CEGr_Diff}, $\underline{D}_i\nu_j=-W_{ij}$ on $\Gamma$, and \eqref{E:HW_Bound}.
  By these relations and \eqref{E:CG2_Bo}, we obtain \eqref{E:CEL_1d}.
  We also observe that
  \begin{align*}
    \Delta(d^2\bar{\eta}) = 2|\bar{\nu}|^2\bar{\eta}+2d\{(\mathrm{div}\,\bar{\nu})\bar{\eta}+2\bar{\nu}\cdot\nabla\bar{\eta}\}+d^2\Delta\bar{\eta} = 2\bar{\eta}+2d(\mathrm{div}\,\bar{\nu})\bar{\eta}+d^2\Delta\bar{\eta}
  \end{align*}
  in $\overline{N}$ by $\nabla d=\bar{\nu}$ in $\overline{N}$, $|\nu|=1$ on $\Gamma$, and \eqref{E:CEGr_TN}.
  We apply \eqref{Pf_CEL:div_nu} to the above equality and then use \eqref{E:HW_Bound}, \eqref{E:CG2_Bo}, and $|d|\leq\delta$ in $\overline{N}$ to get \eqref{E:CEL_2d}.
\end{proof}

\subsection{Moving surface} \label{SS:MovSurf}
For each $t\in[0,T]$, $T>0$, let $\Gamma(t)$ be a given closed, connected, and oriented smooth hypersurface in $\mathbb{R}^n$.
We set $\Gamma_0=\Gamma(0)$ and
\begin{align*}
  S_T = \bigcup_{t\in(0,T]}\Gamma(t)\times\{t\}, \quad \overline{S_T} = (\Gamma_0\times\{0\})\cup S_T,
\end{align*}
and use the same notations as in Section \ref{SS:FixSurf}.

We assume that there exists a smooth mapping $\Phi\colon\Gamma_0\times[0,T]\to\mathbb{R}^n$ such that $\Phi(\cdot,t)$ is a diffeomorphism from $\Gamma_0$ onto $\Gamma(t)$ for each $t\in[0,T]$ with $\Phi(\cdot,0)=\mathrm{Id}$.
Then the unit outward normal vector field $\nu$ and the principal curvatures $\kappa_1,\dots,\kappa_{n-1}$ of $\Gamma(t)$ are smooth and thus bounded as functions on $\overline{S_T}$.
Hence we can take a constant $\delta>0$ independent of $t$ and a tubular neighborhood
\begin{align} \label{E:Tubular}
  \overline{N(t)} = \{x\in\mathbb{R}^n \mid -\delta\leq d(x,t)\leq\delta\}
\end{align}
of $\Gamma(t)$ such that for each $x\in\overline{N(t)}$ there exists a unique $\pi(x,t)\in\Gamma(t)$ satisfying \eqref{E:Fermi} and that $d$ and $\pi$ are smooth on $\overline{N_T}$, where
\begin{align*}
  \overline{N_T} = \bigcup_{t\in[0,T]}\overline{N(t)}\times\{t\}.
\end{align*}
Moreover, \eqref{E:Pr_Curv} holds for $(y,t)\in\overline{S_T}$ instead of $y\in\Gamma$ with a constant $c_0$ independent of $t$, and the functions $\underline{D}_iW_{jk}$ are bounded on $\overline{S_T}$.
Therefore, we can apply Lemmas \ref{L:R_Norm}--\ref{L:CE_Lap} on $\Gamma(t)$ for all $t\in[0,T]$ with a constant $c$ independent of $t$.

We define the (scalar) outer normal velocity of $\Gamma(t)$ by
\begin{align*}
  V_\Gamma(\Phi(Y,t),t) = \partial_t\Phi(Y,t)\cdot\nu(\Phi(Y,t),t), \quad (Y,t) \in \Gamma_0\times[0,T].
\end{align*}
Then $V_\Gamma$ is smooth on $\overline{S_T}$.
Note that the evolution of $\Gamma(t)$ (as a subset of $\mathbb{R}^n$) is determined by the normal velocity field $V_\Gamma\nu$, i.e. if we define a mapping $\Phi_\nu\colon\Gamma_0\times[0,T]\to\mathbb{R}^n$ by
\begin{align} \label{E:N_Flow}
  \Phi_\nu(Y,0) = Y, \quad \partial_t\Phi_\nu(Y,t) = (V_\Gamma\nu)(\Phi_\nu(Y,t),t), \quad (Y,t)\in \Gamma_0\times[0,T],
\end{align}
then $\Phi_\nu(\cdot,t)$ is a diffeomorfism from $\Gamma_0$ onto $\Gamma(t)$ for each $t\in[0,T]$.
Moreover,
\begin{align} \label{E:SDdt_Surf}
  \partial_td(y,t) = -V_\Gamma(y,t), \quad (y,t)\in\overline{S_T},
\end{align}
since the signed distance $d$ from $\Gamma(t)$ increases in the direction of $\nu$.
We can also compute the time derivatives of $d$ and $\pi$ in $\overline{N_T}$ as follows.

\begin{lemma} \label{L:dpi_dt}
  For all $(x,t)\in\overline{N_T}$ we have
  \begin{align}
    \partial_td(x,t) &= -\overline{V_\Gamma}(x,t), \label{E:SDdt_TN} \\
    \partial_t\pi(x,t) &= \overline{V_\Gamma}(x,t)\bar{\nu}(x,t)+d(x,t)R(x,t)\overline{\nabla_\Gamma V_\Gamma}(x,t), \label{E:Pidt}
  \end{align}
  where $R$ is the matrix given by \eqref{E:Inv_IdW}.
\end{lemma}

\begin{proof}
  Since $d(\pi(x,t),t)=0$ by $\pi(x,t)\in\Gamma(t)$, we have
  \begin{align*}
    \partial_td(\pi(x,t),t)+\partial_t\pi(x,t)\cdot\nabla d(\pi(x,t),t) = 0.
  \end{align*}
  Also, since $\pi(x,t)=x-d(x,t)\nabla d(x,t)$ by \eqref{E:Fermi}, it follows that
  \begin{align*}
    \partial_t\pi(x,t) = -\partial_td(x,t)\nabla d(x,t)-d(x,t)\partial_t\nabla d(x,t).
  \end{align*}
  Then we see by $\nabla d(\pi(x,t),t)=\nabla d(x,t)$ and $|\nabla d(x,t)|^2=1$ that
  \begin{align*}
    \partial_t\pi(x,t)\cdot\nabla d(\pi(x,t),t) &= -\partial_td(x,t)|\nabla d(x,t)|^2-\frac{1}{2}d(x,t)\partial_t\bigl(|\nabla d(x,t)|^2\bigr) \\
    &= -\partial_td(x,t).
  \end{align*}
  We deduce from the above equalities and \eqref{E:SDdt_Surf} with $y=\pi(x,t)\in\Gamma(t)$ that
  \begin{align*}
    \partial_td(x,t) = \partial_td(\pi(x,t),t) = -V_\Gamma(\pi(x,t),t) = -\overline{V}_\Gamma(x,t).
  \end{align*}
  Hence \eqref{E:SDdt_TN} holds.
  Moreover, we see by this equality and \eqref{E:CEGr_TN} that
  \begin{align*}
    \partial_t\nabla d(x,t) = \nabla\partial_td(x,t) = -\nabla\Bigl(\overline{V}_\Gamma(x,t)\Bigr) = -R(x,t)\overline{\nabla_\Gamma V_\Gamma}(x,t).
  \end{align*}
  Hence \eqref{E:Pidt} follows from the above equalities and \eqref{E:Fermi}.
\end{proof}

Let $\eta$ be a function on $\overline{S_T}$.
We define the normal time derivative of $\eta$ by
\begin{align} \label{E:NTime_Der}
  \partial^\circ\eta(\Phi_\nu(Y,t),t) = \frac{\partial}{\partial t}\Bigl(\eta(\Phi_\nu(Y,t),t)\Bigr), \quad (Y,t)\in\Gamma_0\times[0,T],
\end{align}
where $\Phi_\nu$ is the mapping given by \eqref{E:N_Flow}.
Note that
\begin{align*}
  \partial^\circ\eta(y,t) = \partial_t\tilde{\eta}(y,t)+(V_\Gamma\nu)(y,t)\cdot\nabla\tilde{\eta}(y,t), \quad (y,t)\in\overline{S_T}
\end{align*}
for any extension $\tilde{\eta}$ of $\eta$ to $\overline{N_T}$.
In particular, setting $\bar{\eta}(x,t)=\eta(\pi(x,t),t)$ for $(x,t)\in\overline{N_T}$, which is the constant extension of $\eta$ in the normal direction of $\Gamma(t)$, we have
\begin{align} \label{E:CEdt_Surf}
  \partial_t\bar{\eta}(y,t) = \partial^\circ\eta(y,t), \quad (y,t)\in\overline{S_T}
\end{align}
by $\nu\cdot\nabla\bar{\eta}=\nu\cdot\nabla_\Gamma\eta=0$ on $\overline{S_T}$ (see \eqref{E:TGr_Tang} and \eqref{E:CEGr_Surf}).

Let us give auxiliary results related to the normal time derivative.

\begin{lemma} \label{L:Nu_NTD}
  For all $(y,t)\in\overline{S_T}$ we have
  \begin{align} \label{E:Nu_NTD}
    \partial^\circ\nu(y,t) = -\nabla_\Gamma V_\Gamma(y,t).
  \end{align}
\end{lemma}

\begin{proof}
  Since $\partial_t\bar{\nu}=\partial_t\nabla d=\nabla\partial_td$ in $\overline{N_T}$ by \eqref{E:Fermi}, we apply \eqref{E:SDdt_TN} to this equality, set $x=y\in\Gamma(t)$, and use \eqref{E:CEGr_Surf} and \eqref{E:CEdt_Surf} to get \eqref{E:Nu_NTD}.
\end{proof}

\begin{lemma} \label{L:CEdt_TN}
  Let $\eta$ be a function on $\overline{S_T}$ and $\bar{\eta}$ be its constant extension.
  Then
  \begin{align} \label{E:CEdt_TN}
    \partial_t\bar{\eta}(x,t) = \overline{\partial^\circ\eta}(x,t)+d(x,t)R(x,t)\overline{\nabla_\Gamma V_\Gamma}(x,t)\cdot\overline{\nabla_\Gamma\eta}(x,t)
  \end{align}
  for all $(x,t)\in\overline{N_T}$.
  Moreover,
  \begin{align}
    |\partial_t\bar{\eta}(x,t)| &\leq c\left(\Bigl|\overline{\partial^\circ\eta}(x,t)\Bigr|+\Bigl|\overline{\nabla_\Gamma\eta}(x,t)\Bigr|\right), \label{E:CEdt_Bound} \\
    \Bigl|\partial_t\bar{\eta}(x,t)-\overline{\partial^\circ\eta}(x,t)\Bigr| &\leq c|d(x)|\Bigl|\overline{\nabla_\Gamma\eta}(x)\Bigr|. \label{E:CEdt_Diff}
  \end{align}
\end{lemma}

\begin{proof}
  We differentiate $\bar{\eta}(x,t)=\bar{\eta}(\pi(x,t),t)$ with respect to $t$ and use \eqref{E:CEGr_Surf} and \eqref{E:CEdt_Surf} with $y=\pi(x,t)\in\Gamma(t)$ to find that
  \begin{align*}
    \partial_t\bar{\eta}(x,t) &= \partial_t\bar{\eta}(\pi(x,t),t)+\partial_t\pi(x,t)\cdot\nabla\bar{\eta}(\pi(x,t),t) \\
    &= \partial^\circ\eta(\pi(x,t),t)+\partial_t\pi(x,t)\cdot\nabla_\Gamma\eta(\pi(x,t),t) \\
    &= \overline{\partial^\circ\eta}(x,t)+\partial_t\pi(x,t)\cdot\overline{\nabla_\Gamma\eta}(x,t).
  \end{align*}
  Then we apply \eqref{E:Pidt} and use \eqref{E:TGr_Tang} to get \eqref{E:CEdt_TN}.
  Also, \eqref{E:CEdt_Bound} and \eqref{E:CEdt_Diff} follow from \eqref{E:R_Norm}, \eqref{E:CEdt_TN}, and the smoothness of $V_\Gamma$ on $\overline{S_T}$.
\end{proof}

\subsection{Moving thin domain} \label{SS:Thin}
Let $g_0$ and $g_1$ be smooth functions on $\overline{S_T}$.
We set $g=g_1-g_0$ on $\overline{S_T}$ and assume that there exists a constant $c>0$ such that
\begin{align} \label{E:G_inf}
  g(y,t) \geq c \quad\text{for all}\quad (y,t) \in \overline{S_T}.
\end{align}
For $\varepsilon>0$, we define a moving thin domain $\Omega_\varepsilon(t)$ by
\begin{align*}
  \Omega_\varepsilon(t) = \{y+r\nu(y,t) \mid y\in\Gamma(t), \, \varepsilon g_0(y,t)<r<\varepsilon g_1(y,t)\}, \quad t\in[0,T]
\end{align*}
and its inner and outer boundaries $\Gamma_\varepsilon^0(t)$ and $\Gamma_\varepsilon^1(t)$ by
\begin{align} \label{E:Def_Bdry}
  \Gamma_\varepsilon^i(t) = \{y+\varepsilon g_i(y,t)\nu(y,t) \mid y\in\Gamma(t)\}, \quad t\in[0,T], \, i=0,1.
\end{align}
We denote by $\partial\Omega_\varepsilon(t)=\Gamma_\varepsilon^0(t)\cup\Gamma_\varepsilon^1(t)$ the whole boundary of $\Omega_\varepsilon(t)$ and set
\begin{gather*}
  Q_{\varepsilon,T} = \bigcup_{t\in(0,T]}\Omega_\varepsilon(t)\times\{t\}, \quad \partial_\ell Q_{\varepsilon,T} = \bigcup_{t\in(0,T]}\partial\Omega_\varepsilon(t)\times\{t\}, \\
  \overline{Q_{\varepsilon,T}} = \bigcup_{t\in[0,T]}\overline{\Omega_\varepsilon(t)}\times\{t\} = \Bigl(\overline{\Omega_\varepsilon(0)}\times\{0\}\Bigr)\cup Q_{\varepsilon,T}\cup\partial_\ell Q_{\varepsilon,T}.
\end{gather*}
Since $g_0$ and $g_1$ are smooth and thus bounded on $\overline{S_T}$, we may take a sufficiently small $\varepsilon_0\in(0,1)$ so that $\varepsilon|g_i|\leq\delta$ on $\overline{S_T}$ for all $\varepsilon\in(0,\varepsilon_0)$ and $i=0,1$, where $\delta>0$ is the time-independent constant appearing in \eqref{E:Tubular}.
Hence for all $\varepsilon\in(0,\varepsilon_0)$ and $t\in[0,T]$ we have $\overline{\Omega_\varepsilon(t)}\subset\overline{N(t)}$ and we can use the lemmas in the previous subsections on $\overline{\Omega_\varepsilon(t)}$ with constants independent of $t$ and $\varepsilon$.
In what follows, we assume $\varepsilon\in(0,\varepsilon_0)$ and denote by $c$ a general positive constant independent of $t$ and $\varepsilon$ (but possibly depending on $T$).
Also, we frequently use the inequality $0<\varepsilon<1$ in the sequel without mention.

Let $\nu_\varepsilon$ and $V_\varepsilon$ be the unit outward normal vector field and the (scalar) outer normal velocity of $\partial\Omega_\varepsilon(t)$.
We express them in terms of functions on $\Gamma(t)$ below.
Let
\begin{align} \label{E:TD_AuxVec}
  \tau_\varepsilon^i(y,t) = \{I_n-\varepsilon g_i(y,t)W(y,t)\}^{-1}\nabla_\Gamma g_i(y,t), \quad (y,t)\in\overline{S_T}, \, i=0,1.
\end{align}
Note that, since $W$ is symmetric and $W\nu=0$ on $\overline{S_T}$,
\begin{align} \label{E:TD_AVTan}
  \tau_\varepsilon^i\cdot\nu = \nabla_\Gamma g_i\cdot\{(I_n-\varepsilon g_iW)^{-1}\nu\} = \nabla_\Gamma g_i\cdot\nu = 0 \quad\text{on}\quad \overline{S_T}.
\end{align}
Also, by \eqref{E:R_Norm} with $x=y+\varepsilon g_i(y,t)\nu(y,t)\in\Gamma_\varepsilon^i(t)$ and the smoothness of $g_i$ on $\overline{S_T}$,
\begin{align} \label{E:TD_AVBo}
  |\tau_\varepsilon^i(y,t)| \leq c, \quad |\tau_\varepsilon^i(y,t)-\nabla_\Gamma g_i(y,t)| \leq c\varepsilon, \quad (y,t) \in \overline{S_T}, \, i=0,1
\end{align}
with a constant $c>0$ independent of $\varepsilon$.
As in the previous subsections, we denote by $\bar{\eta}$ the constant extension of a function $\eta$ on $\overline{S_T}$ in the normal direction of $\Gamma(t)$.

\begin{lemma} \label{L:MTD_UON}
  For $i=0,1$, $t\in[0,T]$, and $x\in\Gamma_\varepsilon(t)$, we have
  \begin{align} \label{E:MTD_UON}
    \nu_\varepsilon(x,t) = \frac{(-1)^{i+1}}{\sqrt{1+\varepsilon^2|\bar{\tau}_\varepsilon^i(x,t)|^2}}\{\bar{\nu}(x,t)-\varepsilon\bar{\tau}_\varepsilon^i(x,t)\}.
  \end{align}
\end{lemma}

\begin{proof}
  We refer to \cite[Lemma 3.9]{Miu22_01} for the proof.
  Note that the paper \cite{Miu22_01} deals with a fixed surface in $\mathbb{R}^3$, but the proof given there is applicable to our case.
\end{proof}

\begin{lemma} \label{L:MTD_ONV}
  For $i=0,1$, $t\in[0,T]$, and $x\in\Gamma_\varepsilon^i(t)$, we have
  \begin{align} \label{E:MTD_ONV}
    V_\varepsilon(x,t) = \frac{(-1)^{i+1}}{\sqrt{1+\varepsilon^2|\bar{\tau}_\varepsilon^i(x,t)|^2}}\Bigl(\overline{V_\Gamma}+\varepsilon\,\overline{\partial^\circ g_i}+\varepsilon^2\bar{g}_i\bar{\tau}_\varepsilon^i\cdot\overline{\nabla_\Gamma V_\Gamma}\Bigr)(x,t).
  \end{align}
\end{lemma}

\begin{proof}
  Fix $i=0,1$.
  For $X\in\Gamma_\varepsilon^i(0)$ and $t\in[0,T]$, we set $Y=\pi(X,0)\in\Gamma(0)$ and
  \begin{align} \label{Pf_ONV:Map}
    \Phi_\varepsilon^i(X,t) = \Phi_\nu(Y,t)+\varepsilon g_i(\Phi_\nu(Y,t),t)\nu(\Phi_\nu(Y,t),t),
  \end{align}
  where $\Phi_\nu$ is the mapping given by \eqref{E:N_Flow}.
  Then since $\Phi_\nu(\cdot,t)$ is a diffeomorphism from $\Gamma_0$ onto $\Gamma(t)$ for each $t\in[0,T]$, and since the relation \eqref{E:Fermi} holds for all $(x,t)\in\overline{N_T}$ and $\Gamma_\varepsilon^i(t)$ is given by \eqref{E:Def_Bdry}, we observe that $\Phi_\varepsilon^i(\cdot,t)$ is a diffeomorphism from $\Gamma_\varepsilon^i(0)$ onto $\Gamma_\varepsilon^i(t)$ for each $t\in[0,T]$.
  Hence the outer normal velocity of $\Gamma_\varepsilon^i(t)$ is given by
  \begin{align*}
    V_\varepsilon(\Phi_\varepsilon^i(X,t),t) = \nu_\varepsilon(\Phi_\varepsilon^i(X,t),t)\cdot\partial_t\Phi_\varepsilon^i(X,t), \quad (X,t) \in \Gamma_\varepsilon(0)\times[0,T].
  \end{align*}
  Now let $x=\Phi_\varepsilon^i(X,t)\in\Gamma_\varepsilon^i(t)$ and $y=\pi(x,t)=\Phi_\nu(Y,t)\in\Gamma(t)$.
  We differentiate \eqref{Pf_ONV:Map} with respect to $t$ and use \eqref{E:N_Flow}, \eqref{E:NTime_Der}, and \eqref{E:Nu_NTD} to get
  \begin{align} \label{Pf_ONV:dt_Phiei}
    \partial_t\Phi_\varepsilon^i(X,t) = (V_\Gamma\nu)(y,t)+\varepsilon\{\partial^\circ g_i(y,t)\nu(y,t)-g_i(y,t)\nabla_\Gamma V_\Gamma(y,t)\}.
  \end{align}
  Hence, noting that $\bar{\eta}(x,t)=\eta(y,t)$ for a function $\eta$ on $\overline{S_T}$, we take the inner product of \eqref{E:MTD_UON} and \eqref{Pf_ONV:dt_Phiei} and use $|\nu|=1$ and $\tau_\varepsilon^i\cdot\nu=\nabla_\Gamma V_\Gamma\cdot\nu=0$ on $\overline{S_T}$ to obtain \eqref{E:MTD_ONV}.
\end{proof}

\section{Uniform a priori estimate for a solution to the thin domain problem} \label{S:Bulk}
The purpose of this section is to show a uniform a priori estimate for a classical solution to the heat equation in the moving thin domain.
To this end, we first consider an initial-boundary value problem of a parabolic equation
\begin{align} \label{E:Para_MTD}
  \left\{
  \begin{alignedat}{3}
    \partial_t\chi^\varepsilon-\sum_{i,j=1}^na_{ij}^\varepsilon\partial_i\partial_j\chi^\varepsilon+\sum_{i=1}^nb_i^\varepsilon\partial_i\chi^\varepsilon+c^\varepsilon\chi^\varepsilon &= f^\varepsilon &\quad &\text{in} &\quad &Q_{\varepsilon,T}, \\
    \partial_{\nu_\varepsilon}\chi^\varepsilon+\beta^\varepsilon\chi^\varepsilon &= \psi^\varepsilon &\quad &\text{on} &\quad &\partial_\ell Q_{\varepsilon,T}, \\
    \chi^\varepsilon|_{t=0} &= \chi_0^\varepsilon &\quad &\text{in} &\quad &\Omega_\varepsilon(0).
  \end{alignedat}
  \right.
\end{align}
Here $a_{ij}^\varepsilon$, $b_\varepsilon^i$, $c^\varepsilon$, and $f^\varepsilon$ are given functions defined on $Q_{\varepsilon,T}$, and $\beta^\varepsilon$ and $\psi^\varepsilon$ are given functions defined on $\partial_\ell Q_{\varepsilon,T}$, and $\chi_0^\varepsilon$ is a given initial data defined on $\overline{\Omega_\varepsilon(0)}$.
Also, we write $\partial_{\nu_\varepsilon}=\nu_\varepsilon\cdot\nabla$ for the outer normal derivative on $\partial_\ell Q_{\varepsilon,T}$.
We assume that
\begin{align} \label{E:Weak_Para}
  \begin{gathered}
    a_{ij}^\varepsilon(x,t) = a_{ji}^\varepsilon(x,t) \quad\text{for all}\quad (x,t)\in Q_{\varepsilon,T}, \, i,j=1,\dots,n, \\
    \sum_{i,j=1}^na_{ij}^\varepsilon(x,t)\xi_i\xi_j \geq 0 \quad\text{for all}\quad (x,t)\in Q_{\varepsilon,T}, \, \xi=(\xi_1,\dots,\xi_n)^T\in\mathbb{R}^n
  \end{gathered}
\end{align}
and say that $\chi^\varepsilon$ is a classical solution to \eqref{E:Para_MTD} if
\begin{align*}
  \chi^\varepsilon \in C(\overline{Q_{\varepsilon,T}}), \quad \partial_i\chi^\varepsilon \in C(Q_{\varepsilon,T}\cup\partial_\ell Q_{\varepsilon,T}), \quad \partial_t\chi^\varepsilon, \partial_i\partial_j\chi^\varepsilon \in C(Q_{\varepsilon,T})
\end{align*}
for all $i,j=1,\dots,n$ and $\chi^\varepsilon$ satisfies \eqref{E:Para_MTD} at each point of $\overline{Q_{\varepsilon,T}}$.

\begin{theorem} \label{T:Uni_Para}
  Let $\varepsilon\in(0,\varepsilon_0)$ and $\chi^\varepsilon$ be a classical solution to \eqref{E:Para_MTD}.
  Suppose that \eqref{E:Weak_Para} holds and there exist constants $c_1,c_2>0$ independent of $\varepsilon$ such that
  \begin{align} \label{E:UP_Corf}
    \max_{i,j=1,\dots,n}|a_{ij}^\varepsilon(x,t)| \leq c_1, \quad \max_{i=1,\dots,n}|b_i^\varepsilon(x,t)| \leq c_1, \quad |c^\varepsilon(x,t)| \leq c_1
  \end{align}
  for all $(x,t)\in Q_{\varepsilon,T}$ and
  \begin{align} \label{E:UP_Bdry}
    \beta^\varepsilon(x,t) \geq -c_2\varepsilon \quad\text{for all}\quad (x,t) \in \partial_\ell Q_{\varepsilon,T}.
  \end{align}
  Then there exists a constant $c_T>0$ depending on $T$ but independent of $\varepsilon$, $\chi_0^\varepsilon$, $f^\varepsilon$, $\psi^\varepsilon$, and $\chi^\varepsilon$ such that
  \begin{align} \label{E:Uni_Para}
    \sup_{\overline{Q_{\varepsilon,T}}}|\chi^\varepsilon| \leq c_T\left(\sup_{\overline{\Omega_\varepsilon(0)}}|\chi_0^\varepsilon|+\sup_{Q_{\varepsilon,T}}|f^\varepsilon|+\frac{1}{\varepsilon}\sup_{\partial_\ell Q_{\varepsilon,T}}|\psi^\varepsilon|\right)
  \end{align}
  provided that the right-hand side is finite.
\end{theorem}

Note that we do not assume that the coefficients and given data in \eqref{E:Para_MTD} are continuous on their domains (but $\chi_0^\varepsilon=\chi^\varepsilon|_{t=0}$ must be continuous on $\overline{\Omega_\varepsilon(0)}$ by the continuity of $\chi^\varepsilon$).
This is not important in this paper, but it may be useful for other problems.

To prove Theorem \ref{T:Uni_Para}, we follow a standard argument of the proof of an a priori estimate based on the maximum principle (see e.g. \cite{LaSoUr68,Lie96}), but it is necessary to analyze carefully the dependence of coefficients and constants on $\varepsilon$.
We first introduce an auxiliary function and then give the proof of Theorem \ref{T:Uni_Para}.
Recall that we write $\bar{\eta}$ for the constant extension of a function $\eta$ on $\overline{S_T}$ in the normal direction of $\Gamma(t)$.

\begin{lemma} \label{L:Uni_Aux}
  For $(x,t)\in\overline{Q_{\varepsilon,T}}$ we define
  \begin{align} \label{E:UAux_Def}
    \sigma_\varepsilon(x,t) = \{d(x,t)-\varepsilon\bar{g}_0(x,t)\}\{d(x,t)-\varepsilon\bar{g}_1(x,t)\}.
  \end{align}
  Then there exists a constant $c_3>0$ independent of $\varepsilon$ such that
  \begin{align} \label{E:UAux_In}
      |\sigma_\varepsilon(x,t)| \leq c_3\varepsilon^2, \quad |\partial_t\sigma_\varepsilon(x,t)| \leq c_3\varepsilon, \quad |\nabla\sigma_\varepsilon(x,t)| \leq c_3\varepsilon, \quad |\nabla^2\sigma_\varepsilon(x,t)| \leq c_3
  \end{align}
  for all $(x,t)\in\overline{Q_{\varepsilon,T}}$ and
  \begin{align} \label{E:UAux_Bdry}
    \nu_\varepsilon(x,t)\cdot\nabla \sigma_\varepsilon(x,t) \geq c_3\varepsilon \quad\text{for all}\quad (x,t)\in\partial_\ell Q_{\varepsilon,T}.
  \end{align}
\end{lemma}

\begin{proof}
  We see that \eqref{E:UAux_In} holds by \eqref{E:CEGr_Bound}, \eqref{E:CG2_Bo}, \eqref{E:CEdt_Bound}, the smoothness of $d$ on $\overline{N_T}$ and of $g_0$ and $g_1$ on $\overline{S_T}$, and $|d|\leq c\varepsilon$ in $\overline{Q_{\varepsilon,T}}$.
  Also, since
  \begin{align*}
    \nabla \sigma_\varepsilon = (d-\varepsilon\bar{g}_1)\Bigl(\bar{\nu}-\varepsilon R\,\overline{\nabla_\Gamma g_0}\Bigr)+(d-\varepsilon\bar{g}_0)\Bigl(\bar{\nu}-\varepsilon R\,\overline{\nabla_\Gamma g_1}\Bigr) \quad\text{in}\quad \overline{N_T}
  \end{align*}
  by $\nabla d=\bar{\nu}$ in $\overline{N_T}$ and \eqref{E:CEGr_TN}, where $R$ is given by \eqref{E:Inv_IdW}, and since $d(x,t)=\varepsilon\bar{g}_i(x,t)$ for each $x\in\Gamma_\varepsilon^i(t)$, $t\in[0,T]$, and $i=0,1$ by \eqref{E:Def_Bdry}, we have
  \begin{align*}
    \nabla \sigma_\varepsilon(x,t) = (-1)^{i+1}\varepsilon\bar{g}(x,t)\{\bar{\nu}(x,t)-\varepsilon\bar{\tau}_\varepsilon^i(x,t)\}, \quad x\in\Gamma_\varepsilon^i(t), \, t\in[0,T], \, i=0,1,
  \end{align*}
  where $g=g_1-g_0$ and $\tau_\varepsilon^i$ is given by \eqref{E:TD_AuxVec} on $\overline{S_T}$.
  Using this equality, \eqref{E:MTD_UON}, and $\tau_\varepsilon^i\cdot\nu=0$ on $\overline{S_T}$, and then applying \eqref{E:G_inf}, we find that
  \begin{align*}
    \nu_\varepsilon(x,t)\cdot\nabla \sigma_\varepsilon(x,t) = \varepsilon\bar{g}(x,t)\sqrt{1+\varepsilon^2|\bar{\tau}_\varepsilon^i(x,t)|^2} \geq \varepsilon\bar{g}(x,t) \geq c\varepsilon
  \end{align*}
  for all $x\in\Gamma_\varepsilon^i(t)$, $t\in[0,T]$, and $i=0,1$.
  Hence \eqref{E:UAux_Bdry} is valid.
\end{proof}

\begin{proof}[Proof of Theorem \ref{T:Uni_Para}]
  Let $\chi^\varepsilon$ be a classical solution to \eqref{E:Para_MTD} and $\sigma_\varepsilon$ be the function given by \eqref{E:UAux_Def}.
  Also, let $c_1$ and $c_2$ be the constants appearing in \eqref{E:UP_Corf} and \eqref{E:UP_Bdry}, and $c_3$ be the constant given in Lemma \ref{L:Uni_Aux}.
  We define
  \begin{align*}
    c_4 = \frac{c_2+1}{c_3}, \quad \zeta^\varepsilon(x,t) = e^{-c_4\sigma_\varepsilon(x,t)}\chi^\varepsilon(x,t), \quad (x,t) \in \overline{Q_{\varepsilon,T}}.
  \end{align*}
  Then, noting that $\sigma_\varepsilon=0$ on $\partial_\ell Q_{\varepsilon,T}$, we observe that $\zeta^\varepsilon$ satisfies
  \begin{align} \label{Pf_UP:Zeta}
    \left\{
    \begin{alignedat}{3}
      \partial_t\zeta^\varepsilon-\sum_{i,j=1}^na_{ij}^\varepsilon\partial_i\partial_j\zeta^\varepsilon+\sum_{i=1}^nB_i^\varepsilon\partial_i\zeta^\varepsilon+C^\varepsilon\zeta^\varepsilon &= e^{-c_4\sigma_\varepsilon}f^\varepsilon &\quad &\text{in} &\quad &Q_{\varepsilon,T}, \\
      \partial_{\nu_\varepsilon}\zeta^\varepsilon+\gamma^\varepsilon\zeta^\varepsilon &= \psi^\varepsilon &\quad &\text{on} &\quad &\partial_\ell Q_{\varepsilon,T}, \\
      \zeta^\varepsilon|_{t=0} &= e^{-c_4\sigma_\varepsilon(\cdot,0)}\chi_0^\varepsilon &\quad &\text{in} &\quad &\Omega_\varepsilon(0),
    \end{alignedat}
    \right.
  \end{align}
  where the functions $B_\varepsilon^i$ and $C^\varepsilon$ on $Q_{\varepsilon,T}$ and $\gamma^\varepsilon$ on $\partial_\ell Q_{\varepsilon,T}$ are given by
  \begin{align*}
    B_i^\varepsilon &= -2c_4\sum_{j=1}^na_{ij}^\varepsilon\partial_j\sigma_\varepsilon+b_i^\varepsilon, \quad i=1,\dots,n, \\
    C^\varepsilon &= c_4\partial_t\sigma_\varepsilon-c_4\sum_{i,j=1}^na_{ij}^\varepsilon\partial_i\partial_j\sigma_\varepsilon-c_4^2\sum_{i,j=1}^na_{ij}^\varepsilon(\partial_i\sigma_\varepsilon)(\partial_j\sigma_\varepsilon)+c_4\sum_{i=1}^nb_i^\varepsilon\partial_i\sigma_\varepsilon+c^\varepsilon
  \end{align*}
  on $Q_{\varepsilon,T}$ and $\gamma^\varepsilon=c_4(\nu^\varepsilon\cdot\nabla \sigma_\varepsilon)+\beta^\varepsilon$ on $\partial_\ell Q_{\varepsilon,T}$.
  Moreover, $|C^\varepsilon|\leq c_5$ in $Q_{\varepsilon,T}$ by \eqref{E:UP_Corf} and \eqref{E:UAux_In}, where $c_5>0$ is a constant depending only on $c_1$, $c_2$, and $c_3$ and thus independent of $\varepsilon$.
  Also, it follows from \eqref{E:UP_Bdry}, \eqref{E:UAux_Bdry}, and $c_4=(c_2+1)/c_3$ that
  \begin{align} \label{Pf_UP:Ga_Ineq}
    \gamma^\varepsilon \geq c_4\cdot c_3\varepsilon-c_2\varepsilon = \varepsilon > 0 \quad\text{on}\quad \partial_\ell Q_{\varepsilon,T}.
  \end{align}
  Thus, setting
  \begin{align} \label{Pf_UP:C_hat}
    \widehat{C}^\varepsilon(x,t) = C^\varepsilon(x,t)+c_5+1 \geq -|C^\varepsilon(x,t)|+c_5+1 \geq 1
  \end{align}
  for $(x,t)\in Q_{\varepsilon,T}$ and
  \begin{align*}
    Z_\pm^\varepsilon(x,t) = \pm e^{-(c_5+1)t}\zeta^\varepsilon(x,t)-\left(\sup_{\overline{\Omega_\varepsilon(0)}}|e^{-c_4\sigma_\varepsilon(\cdot,0)}\chi_0^\varepsilon|+\sup_{Q_{\varepsilon,T}}|e^{-c_4\sigma_\varepsilon}f^\varepsilon|+\frac{1}{\varepsilon}\sup_{\partial_\ell Q_{\varepsilon,T}}|\psi^\varepsilon|\right)
  \end{align*}
  for $(x,t)\in\overline{Q_{\varepsilon,T}}$, we see that $Z_\pm^\varepsilon$ satisfies
  \begin{align} \label{Pf_UP:Z_Ineq}
    \left\{
    \begin{alignedat}{3}
      \partial_tZ_\pm^\varepsilon-\sum_{i,j=1}^na_{ij}^\varepsilon\partial_i\partial_jZ_\pm^\varepsilon+\sum_{i=1}^nB_i^\varepsilon\partial_iZ_\pm^\varepsilon+\widehat{C}^\varepsilon Z_\pm^\varepsilon &\leq 0 &\quad &\text{in} &\quad &Q_{\varepsilon,T}, \\
      \partial_{\nu_\varepsilon}Z_\pm^\varepsilon+\gamma^\varepsilon Z_\pm^\varepsilon &\leq 0 &\quad &\text{on} &\quad &\partial_\ell Q_{\varepsilon,T}, \\
      Z_\pm^\varepsilon|_{t=0} &\leq 0 &\quad &\text{in} &\quad &\Omega_\varepsilon(0).
    \end{alignedat}
    \right.
  \end{align}
  Note that, due to \eqref{Pf_UP:Ga_Ineq}, we need to multiply the supremum of $|\psi^\varepsilon|$ by $\varepsilon^{-1}$ in the definition of $Z_\pm^\varepsilon$ in order to get the second inequality of \eqref{Pf_UP:Z_Ineq}.
  Then, as in the proof of the maximum principle (see e.g. \cite{LaSoUr68,Lie96}), we can show that the maximum of $Z_\pm^\varepsilon$ on $\overline{Q_{\varepsilon,T}}$ must be nonpositive by using \eqref{E:Weak_Para} and \eqref{Pf_UP:Ga_Ineq}--\eqref{Pf_UP:Z_Ineq} and by noting that $\partial_{\nu_\varepsilon}Z_\pm^\varepsilon(x,t)\geq0$ if $Z_\pm^\varepsilon$ attains its maximum at $(x,t)\in\partial_\ell Q_{\varepsilon,T}$.
  Hence $Z_\pm^\varepsilon\leq 0$, i.e.
  \begin{align*}
    |\zeta^\varepsilon| \leq e^{(c_5+1)t}\left(\sup_{\overline{\Omega_\varepsilon(0)}}|e^{-c_4\sigma_\varepsilon(\cdot,0)}\chi_0^\varepsilon|+\sup_{Q_{\varepsilon,T}}|e^{-c_4\sigma_\varepsilon}f^\varepsilon|+\frac{1}{\varepsilon}\sup_{\partial_\ell Q_{\varepsilon,T}}|\psi^\varepsilon|\right)
  \end{align*}
  in $\overline{Q_{\varepsilon,T}}$.
  By this inequality, $\chi^\varepsilon=e^{c_4\sigma_\varepsilon}\zeta^\varepsilon$ in $\overline{Q_{\varepsilon,T}}$, and \eqref{E:UAux_In}, we conclude that \eqref{E:Uni_Para} holds with $c_T=c_6e^{(c_5+1)T}$ with a constant $c_6>0$ independent of $\varepsilon$.
\end{proof}

\begin{remark} \label{R:Uni_Para}
  By the above proof, one may expect to remove the factor $\varepsilon^{-1}$ in \eqref{E:Uni_Para} if \eqref{E:UP_Bdry} is replaced by $\beta^\varepsilon\geq-c$ on $\partial_\ell Q_{\varepsilon,T}$ with a constant $c>0$ independent of $\varepsilon$, but we must have another constant growing as $\varepsilon\to0$ in that case.
  Indeed, when $\beta^\varepsilon\geq-c$ on $\partial_\ell Q_{\varepsilon,T}$, we need to employ $\varepsilon^{-1}\sigma_\varepsilon$ instead of $\sigma_\varepsilon$ in the proof of Theorem \ref{T:Uni_Para} in order to get $\gamma^\varepsilon\geq 1$ on $\partial_\ell Q_{\varepsilon,T}$ instead of \eqref{Pf_UP:Ga_Ineq} for the coefficient $\gamma^\varepsilon$ in \eqref{Pf_UP:Zeta}.
  However, by taking $\varepsilon^{-1}\sigma_\varepsilon$ we must have $C^\varepsilon\geq-c\varepsilon^{-1}$ in $Q_{\varepsilon,T}$ for the coefficient $C^\varepsilon$ in \eqref{Pf_UP:Zeta} even if we carefully compute the derivatives of $\sigma_\varepsilon$.
  Hence, if we intend to remove $\varepsilon^{-1}$ in \eqref{E:Uni_Para}, then we need to replace $c_T=c_6e^{(c_5+1)T}$ by $c_6e^{(c_5+1)T/\varepsilon}$ which grows as $\varepsilon\to0$ much faster than $\varepsilon^{-1}$.
\end{remark}

Now we consider the heat equation with source terms
\begin{align} \label{E:Heat_Gen}
  \left\{
  \begin{alignedat}{3}
    \partial_t\rho^\varepsilon-k_d\Delta\rho^\varepsilon &= f^\varepsilon &\quad &\text{in} &\quad &Q_{\varepsilon,T}, \\
    \partial_{\nu_\varepsilon}\rho^\varepsilon+k_d^{-1}V_\varepsilon\rho^\varepsilon &= \psi^\varepsilon &\quad &\text{on} &\quad &\partial_\ell Q_{\varepsilon,T}, \\
    \rho^\varepsilon|_{t=0} &= \rho_0^\varepsilon &\quad &\text{in} &\quad &\Omega_\varepsilon(0).
  \end{alignedat}
  \right.
\end{align}
In this case, the outer normal velocity $V_\varepsilon$ of $\partial\Omega_\varepsilon(t)$ is of order one with respect to $\varepsilon$ by \eqref{E:MTD_ONV}.
Thus, in order to apply Theorem \ref{T:Uni_Para}, we need to introduce an auxiliary function to eliminate the zeroth order term of $V_\varepsilon$.

\begin{theorem} \label{T:Uni_Heat}
  Let $\varepsilon\in(0,\varepsilon_0)$ and
  \begin{align*}
    \rho_0^\varepsilon \in C(\overline{\Omega_\varepsilon(0)}), \quad f^\varepsilon \in C(Q_{\varepsilon,T}), \quad \psi^\varepsilon \in C(\partial_\ell Q_{\varepsilon,T}).
  \end{align*}
  Also, let $\rho^\varepsilon$ be a classical solution to \eqref{E:Heat_Gen}.
  Then there exists a constant $c_T>0$ depending on $T$ but independent of $\varepsilon$, $\rho_0^\varepsilon$, $f^\varepsilon$, $\psi^\varepsilon$, and $\rho^\varepsilon$ such that
  \begin{align} \label{E:Uni_Heat}
    \|\rho^\varepsilon\|_{C(\overline{Q_{\varepsilon,T}})} \leq c_T\left(\|\rho_0^\varepsilon\|_{C(\overline{\Omega_\varepsilon(0)})}+\|f^\varepsilon\|_{C(Q_{\varepsilon,T})}+\frac{1}{\varepsilon}\|\psi^\varepsilon\|_{C(\partial_\ell Q_{\varepsilon,T})}\right)
  \end{align}
  provided that the right-hand side is finite.
\end{theorem}

\begin{proof}
  Let $\rho^\varepsilon$ be a classical solution to \eqref{E:Heat_Gen}.
  We define
  \begin{align*}
    \lambda(x,t) = -k_d^{-1}d(x,t)\overline{V_\Gamma}(x,t), \quad \chi^\varepsilon(x,t) = e^{-\lambda(x,t)}\rho^\varepsilon(x,t)
  \end{align*}
  for $(x,t)\in\overline{Q_{\varepsilon,T}}$.
  Then $\chi^\varepsilon$ satisfies \eqref{E:Para_MTD} with
  \begin{align*}
    a_{ij}^\varepsilon &= k_d\delta_{ij}, \quad b_i^\varepsilon = -2k_d\partial_i\lambda, \quad c^\varepsilon = \partial_t\lambda-k_d(\Delta\lambda+|\nabla\lambda|^2) \quad\text{in}\quad Q_{\varepsilon,T}, \\
    \beta^\varepsilon &= \nu_\varepsilon\cdot\nabla\lambda+k_d^{-1}V_\varepsilon \quad\text{on}\quad \partial_\ell Q_{\varepsilon,T},
  \end{align*}
  and $f^\varepsilon$, $\psi^\varepsilon$, and $\chi_0^\varepsilon$ replaced by $e^{-\lambda}f^\varepsilon$, $e^{-\lambda}\psi^\varepsilon$, and $e^{-\lambda(\cdot,0)}\rho_0^\varepsilon$.
  Here $\delta_{ij}$ is the Kronecker delta.
  Then, since $d$ and $V_\Gamma$ are smooth on $\overline{N_T}$ and $\overline{S_T}$, we see by \eqref{E:CEGr_Bound}, \eqref{E:CG2_Bo}, and \eqref{E:CEdt_Bound} that \eqref{E:UP_Corf} holds with a constant $c_1>0$ independent of $\varepsilon$.
  Also, since
  \begin{align*}
    \nabla\lambda = -k_d^{-1}\overline{V_\Gamma\nu}-k_d^{-1}d\,\nabla\overline{V_\Gamma} \quad\text{in}\quad \overline{Q_{\varepsilon,T}}
  \end{align*}
  by $\nabla d=\bar{\nu}$ in $\overline{N_T}$, and since $\nu_\varepsilon$ is of the form \eqref{E:MTD_UON}, we have
  \begin{align*}
    \nu_\varepsilon\cdot\nabla\lambda = \frac{(-1)^{i+1}k_d^{-1}}{\sqrt{1+\varepsilon^2|\bar{\tau}_\varepsilon^i|^2}}\Bigl(-\overline{V_\Gamma}+\varepsilon^2\bar{g}_i\bar{\tau}_\varepsilon^i\cdot\nabla\overline{V_\Gamma}\Bigr) \quad\text{on}\quad \Gamma_\varepsilon^i(t)
  \end{align*}
  for $t\in(0,T]$ and $i=0,1$.
  Here $\tau_\varepsilon^i$ is given by \eqref{E:TD_AuxVec} and we also used
  \begin{align*}
    |\nu| = 1, \, \tau_\varepsilon^i\cdot\nu = 0 \quad\text{on}\quad \overline{S_T}, \quad \bar{\nu}\cdot\nabla\overline{V_\Gamma} = 0 \quad\text{in}\quad \overline{N_T}, \quad d = \varepsilon\bar{g}_i \quad\text{on}\quad \Gamma_\varepsilon^i(t),
  \end{align*}
  see \eqref{E:TD_AVTan} for the second equality.
  By the above equality and \eqref{E:MTD_ONV}, we get
  \begin{align*}
    \beta_\varepsilon = \frac{(-1)^{i+1}k_d^{-1}}{\sqrt{1+\varepsilon^2|\bar{\tau}_\varepsilon^i|^2}}\left\{\varepsilon\,\overline{\partial^\circ g_i}+\varepsilon^2\bar{g}_i\bar{\tau}_\varepsilon^i\cdot\Bigl(\nabla\overline{V_\Gamma}+\overline{\nabla_\Gamma V_\Gamma}\Bigr)\right\} \quad\text{on}\quad \Gamma_\varepsilon^i(t)
  \end{align*}
  for $t\in(0,T]$ and $i=0,1$.
  Thus, by \eqref{E:CEGr_Bound}, \eqref{E:TD_AVBo}, and the smoothness of $g_0$, $g_1$, and $V_\Gamma$ on $\overline{S_T}$, we find that $|\beta_\varepsilon|\leq c\varepsilon$ on $\partial_\ell Q_{\varepsilon,T}$ with a constant $c>0$ independent of $\varepsilon$.
  Hence \eqref{E:UP_Bdry} is valid and we can apply Theorem \ref{T:Uni_Para} to obtain
  \begin{align*}
    \sup_{\overline{Q_{\varepsilon,T}}}|\chi^\varepsilon| \leq c_T\left(\sup_{\overline{\Omega_\varepsilon(0)}}|e^{-\lambda(\cdot,0)}\rho_0^\varepsilon|+\sup_{Q_{\varepsilon,T}}|e^{-\lambda}f^\varepsilon|+\frac{1}{\varepsilon}\sup_{\partial_\ell Q_{\varepsilon,T}}|e^{-\lambda}\psi^\varepsilon|\right).
  \end{align*}
  Recall that $\chi^\varepsilon$ satisfies \eqref{E:Para_MTD} with $f^\varepsilon$, $\psi^\varepsilon$, and $\chi_0^\varepsilon$ replaced by $e^{-\lambda}f^\varepsilon$, $e^{-\lambda}\psi^\varepsilon$, and $e^{-\lambda(\cdot,0)}\rho_0^\varepsilon$.
  Now we observe that $\rho^\varepsilon=e^\lambda\chi^\varepsilon$ in $\overline{Q_{\varepsilon,T}}$ and $\lambda=-k_d^{-1}d\,\overline{V_\Gamma}$ is bounded on $\overline{N_T}$ independently of $\varepsilon$.
  Therefore, we obtain \eqref{E:Uni_Heat} by the above estimate for $\chi^\varepsilon$.
\end{proof}

\begin{remark} \label{R:Uni_Heat}
  The idea of the proof of Theorem \ref{T:Uni_Heat} also applies to the problem \eqref{E:Para_MTD} under the conditions \eqref{E:UP_Corf} and $\beta^\varepsilon\geq-c$ on $\partial_\ell Q_{\varepsilon,T}$ instead of \eqref{E:UP_Bdry}, provided that there exists a function $\omega^\varepsilon$ on $\overline{Q_{\varepsilon,T}}$ such that $\omega^\varepsilon$, $\partial_t\omega^\varepsilon$, $\nabla\omega^\varepsilon$, and $\nabla^2\omega^\varepsilon$ are uniformly bounded on $\overline{Q_{\varepsilon,T}}$ and that $\beta^\varepsilon-(-1)^{i+1}\omega^\varepsilon\geq-c\varepsilon$ on $\Gamma_\varepsilon^i(t)$ for all $t\in(0,T]$ and $i=0,1$.
  In this case, for a classical solution $\chi^\varepsilon$ to \eqref{E:Para_MTD} we see that $e^{d\omega^\varepsilon}\chi^\varepsilon$ satisfies a problem similar to \eqref{E:Para_MTD} with coefficients in $Q_{\varepsilon,T}$ satisfying \eqref{E:UP_Corf} and with the coefficient on $\partial_\ell Q_{\varepsilon,T}$ given by
  \begin{align*}
    \hat{\beta}^\varepsilon=\beta^\varepsilon-(\nu^\varepsilon\cdot\bar{\nu})\omega^\varepsilon-\varepsilon\bar{g}_i(\nu^\varepsilon\cdot\nabla\omega^\varepsilon) \quad\text{on}\quad \Gamma_\varepsilon^i(t)
  \end{align*}
  for $t\in(0,T]$ and $i=0,1$.
  Then it follows from \eqref{E:TD_AVTan}--\eqref{E:MTD_UON}, the mean value theorem for $(1+s)^{-1/2}$ with $s\in\mathbb{R}$, and the assumption on $\omega^\varepsilon$ that
  \begin{align*}
    \hat{\beta}^\varepsilon &= \beta^\varepsilon-(-1)^{i+1}\omega^\varepsilon+(-1)^{i+1}\omega^\varepsilon\left(1-\frac{1}{\sqrt{1+\varepsilon^2|\bar{\tau}_\varepsilon^i|^2}}\right)-\varepsilon\bar{g}_i(\nu^\varepsilon\cdot\nabla\omega^\varepsilon) \\
    &\geq \beta^\varepsilon-(-1)^{i+1}\omega^\varepsilon-|\omega^\varepsilon|\left|1-\frac{1}{\sqrt{1+\varepsilon^2|\bar{\tau}_\varepsilon^i|^2}}\right|-c\varepsilon|\nabla\omega^\varepsilon| \geq -c\varepsilon
  \end{align*}
  on $\Gamma_\varepsilon^i(t)$ for all $t\in(0,T]$ and $i=0,1$.
  Hence we can apply Theorem \ref{T:Uni_Para} to $e^{d\omega^\varepsilon}\chi^\varepsilon$ and then obtain \eqref{E:Uni_Para} for $\chi^\varepsilon$ by noting that $d\omega^\varepsilon$ is uniformly bounded on $\overline{Q_{\varepsilon,T}}$.
\end{remark}

\section{Proofs of the main results} \label{S:Pf_Main}
In this section we prove Theorems \ref{T:Error_Spe} and \ref{T:Error_Gen}.
We give the proof of Theorem \ref{T:Error_Gen} below, and Theorem \ref{T:Error_Spe} can be obtained by setting $\nabla_\Gamma g_0=\nabla_\Gamma g_1=0$ on $\overline{S_T}$ in the following proof.
As in the previous sections, we denote by $\bar{\eta}$ for the constant extension of a function $\eta$ on $\overline{S_T}$ in the normal direction of $\Gamma(t)$.

Let $\rho^\varepsilon$ and $\eta$ be classical solutions to \eqref{E:Heat_MTD} and \eqref{E:Limit}.
Also, let
\begin{align} \label{E:PM_zeta}
  \zeta_i = \frac{1}{g}\{\nabla_\Gamma g_i\cdot\nabla_\Gamma\eta-k_d^{-1}(\partial^\circ g_i)\eta+k_d^{-2}g_iV_\Gamma^2\eta\}
\end{align}
on $\overline{S_T}$ and $i=0,1$.
For $(x,t)\in\overline{Q_{\varepsilon,T}}$ we define
\begin{align} \label{E:PM_ee12}
  \begin{aligned}
    \eta_2^\varepsilon(x,t) &= \varepsilon d(x,t)\Bigl(\bar{g}_1\bar{\zeta}_0-\bar{g}_0\bar{\zeta}_1\Bigr)(x,t)+\frac{1}{2}d(x,t)^2\Bigl(\bar{\zeta}_0-\bar{\zeta}_1\Bigr)(x,t), \\
    \rho_\eta^\varepsilon(x,t) &= \bar{\eta}(x,t)-k_d^{-1}d(x,t)\Bigl(\overline{V_\Gamma\eta}\Bigr)(x,t)+\eta_2^\varepsilon(x,t).
  \end{aligned}
\end{align}
We also set
\begin{alignat*}{3}
  f_\eta^\varepsilon &= \partial_t\rho_\eta^\varepsilon-k_d\Delta\rho_\eta^\varepsilon-\bar{f} &\quad &\text{on} &\quad &Q_{\varepsilon,T}, \\
  \psi_\eta^\varepsilon &= \partial_{\nu_\varepsilon}\rho_\eta^\varepsilon+k_d^{-1}V_\varepsilon\rho_\eta^\varepsilon &\quad &\text{on} &\quad &\partial_\ell Q_{\varepsilon,T}
\end{alignat*}
so that $\rho^\varepsilon-\rho_\eta^\varepsilon$ satisfies
\begin{align*}
  \left\{
  \begin{alignedat}{3}
    \partial_t(\rho^\varepsilon-\rho_\eta^\varepsilon)-k_d\Delta(\rho^\varepsilon-\rho_\eta^\varepsilon) &= (f^\varepsilon-\bar{f})-f_\eta^\varepsilon &\quad &\text{in} &\quad &Q_{\varepsilon,T}, \\
    \partial_{\nu_\varepsilon}(\rho^\varepsilon-\rho_\eta^\varepsilon)+k_d^{-1}V_\varepsilon(\rho^\varepsilon-\rho_\eta^\varepsilon) &= -\psi_\eta^\varepsilon &\quad &\text{on} &\quad &\partial_\ell Q_{\varepsilon,T}, \\
    (\rho^\varepsilon-\rho_\eta^\varepsilon)|_{t=0} &= \rho_0^\varepsilon-\rho_\eta^\varepsilon(\cdot,0) &\quad &\text{in} &\quad &\Omega_\varepsilon(0).
  \end{alignedat}
  \right.
\end{align*}
Hence it follows from Theorem \ref{T:Uni_Heat} that
\begin{multline*}
  \|\rho^\varepsilon-\rho_\eta^\varepsilon\|_{C(\overline{Q_{\varepsilon,T}})} \leq c_T\Bigl(\|\rho_0^\varepsilon-\rho_\eta^\varepsilon(\cdot,0)\|_{C(\overline{\Omega_\varepsilon(0)})}+\|f^\varepsilon-\bar{f}\|_{C(Q_{\varepsilon,T})}\Bigr) \\
  +c_T\Bigl(\|f_\eta^\varepsilon\|_{C(Q_{\varepsilon,T})}+\frac{1}{\varepsilon}\|\psi_\eta^\varepsilon\|_{C(\partial_\ell Q_{\varepsilon,T})}\Bigr)
\end{multline*}
with a constant $c_T>0$ depending on $T$ but independent of $\varepsilon$, $\rho^\varepsilon$, $\eta$, and the given data.
By this inequality, \eqref{E:PM_ee12}, and $\eta|_{t=0}=\eta_0$ on $\Gamma(0)$, we obtain
\begin{align} \label{E:PM_Diff}
  \begin{aligned}
    \|\rho^\varepsilon-\bar{\eta}\|_{C(\overline{Q_{\varepsilon,T}})} &\leq c_T\Bigl(\|\rho_0^\varepsilon-\bar{\eta}_0\|_{C(\overline{\Omega_\varepsilon(0)})}+\|f^\varepsilon-\bar{f}\|_{C(Q_{\varepsilon,T})}\Bigr) \\
    &\qquad +c_T\left(\left\|d\Bigl(\overline{V_\Gamma\eta}\Bigr)\right\|_{C(\overline{Q_{\varepsilon,T}})}+\|\eta_2^\varepsilon\|_{C(\overline{Q_{\varepsilon,T}})}\right) \\
    &\qquad +c_T\Bigl(\|f_\eta^\varepsilon\|_{C(Q_{\varepsilon,T})}+\frac{1}{\varepsilon}\|\psi_\eta^\varepsilon\|_{C(\partial_\ell Q_{\varepsilon,T})}\Bigr).
  \end{aligned}
\end{align}
Let us further estimate the right-hand side.
In what follows, we use the notation \eqref{E:C21_Norm} and denote by $c$ a general positive constant independent of $\varepsilon$, $\rho^\varepsilon$, $\eta$, and the given data.
We also frequently use the facts that $g_0$, $g_1$, and $V_\Gamma$ are smooth and thus bounded on $\overline{S_T}$ along with their derivatives and that $g$ satisfies \eqref{E:G_inf} without mention.

First we observe by \eqref{E:PM_ee12} and $|d|\leq c\varepsilon$ in $\overline{Q_{\varepsilon,T}}$ that
\begin{align} \label{E:PM_e12B}
  \left\|d\Bigl(\overline{V_\Gamma\eta}\Bigr)\right\|_{C(\overline{Q_{\varepsilon,T}})}+\|\eta_2^\varepsilon\|_{C(\overline{Q_{\varepsilon,T}})} &\leq c\varepsilon\Bigl(\|\eta\|_{C(\overline{S_T})}+\|\zeta_0\|_{C(\overline{S_T})}+\|\zeta_1\|_{C(\overline{S_T})}\Bigr).
\end{align}
Note that $\overline{S_T}$ includes $\Gamma(0)\times\{0\}$.
Next we consider $f_\eta^\varepsilon$.
We see that
\begin{align} \label{E:PM_dtAp}
  \left|\partial_t\rho_\eta^\varepsilon-\overline{\partial^\circ\eta}-k_d^{-1}\Bigl(\overline{V_\Gamma^2\eta}\Bigr)\right| &\leq c\varepsilon\sum_{\xi=\eta,\zeta_0,\zeta_1}\left(|\bar{\xi}|+\Bigl|\overline{\partial^\circ\xi}\Bigr|+\Bigl|\overline{\nabla_\Gamma\xi}\Bigr|\right)
\end{align}
in $Q_{\varepsilon,T}$ by \eqref{E:SDdt_TN}, \eqref{E:CEdt_Bound}, \eqref{E:CEdt_Diff}, and $|d|\leq c\varepsilon$ in $Q_{\varepsilon,T}$.
Also,
\begin{align} \label{E:PM_LapAp}
  \left|\Delta\rho_\eta^\varepsilon-\overline{\Delta_\Gamma\eta}-k_d^{-1}\Bigl(\overline{V_\Gamma H\eta}\Bigr)-\Bigl(\bar{\zeta}_1-\bar{\zeta}_0\Bigr)\right| &\leq c\varepsilon\sum_{\xi=\eta,\zeta_0,\zeta_1}\left(|\bar{\xi}|+\Bigl|\overline{\nabla_\Gamma\xi}\Bigr|+\Bigl|\overline{\nabla_\Gamma^2\xi}\Bigr|\right)
\end{align}
in $Q_{\varepsilon,T}$ by \eqref{E:CE_Lap}--\eqref{E:CEL_2d} and $|d|\leq c\varepsilon$ in $Q_{\varepsilon,T}$.
Moreover, since
\begin{align*}
  \zeta_1-\zeta_0 = \frac{1}{g}\{\nabla_\Gamma g\cdot\nabla_\Gamma\eta-k_d^{-1}(\partial^\circ g)\eta+k_d^{-2}gV_\Gamma^2\eta\} \quad\text{on}\quad S_T
\end{align*}
by \eqref{E:PM_zeta} and $g_1-g_0=g$ on $S_T$, and since $\eta$ satisfies \eqref{E:Limit}, we have
\begin{align} \label{E:PM_Et_Eq}
  \begin{aligned}
    &\partial^\circ\eta+k_d^{-1}V_\Gamma^2\eta-k_d\{\Delta_\Gamma\eta+k_d^{-1}V_\Gamma H\eta+(\zeta_1-\zeta_0)\} \\
    &\qquad = \frac{1}{g}\{g\partial^\circ\eta+(\partial^\circ g)\eta-gV_\Gamma H\eta-k_d\,g\Delta_\Gamma\eta-k_d\nabla_\Gamma g\cdot\nabla_\Gamma\eta\} \\
    &\qquad = \frac{1}{g}\{\partial^\circ(g\eta)-gV_\Gamma H\eta-k_d\,\mathrm{div}_\Gamma(g\nabla_\Gamma\eta)\} = \frac{gf}{g} = f \quad\text{on}\quad S_T.
  \end{aligned}
\end{align}
Thus, by \eqref{E:PM_dtAp}--\eqref{E:PM_Et_Eq} and $f_\eta^\varepsilon=\partial_t\rho_\eta^\varepsilon-k_d\Delta\rho_\eta^\varepsilon-\bar{f}$ in $Q_{\varepsilon,T}$, we find that
\begin{align*}
  |f_\eta^\varepsilon| &\leq c\varepsilon\sum_{\xi=\eta,\zeta_0,\zeta_1}\left(|\bar{\xi}|+\Bigl|\overline{\partial^\circ\xi}\Bigr|+\Bigl|\overline{\nabla_\Gamma\xi}\Bigr|+\Bigl|\overline{\nabla_\Gamma^2\xi}\Bigr|\right)
\end{align*}
in $Q_{\varepsilon,T}$, which implies that
\begin{align} \label{E:PM_feeB}
  \|f_\eta^\varepsilon\|_{C(Q_{\varepsilon,T})} \leq c\varepsilon\Bigl(\|\eta\|_{C^{2,1}(S_T)}+\|\zeta_0\|_{C^{2,1}(S_T)}+\|\zeta_1\|_{C^{2,1}(S_T)}\Bigr).
\end{align}
Note that here $\Gamma(0)\times\{0\}$ is not included.
Now let us estimate $\psi_\eta^\varepsilon$ on $\partial_\ell Q_{\varepsilon,T}$.
Let $i=0,1$ and $t\in(0,T]$.
Since $\nu_\varepsilon$ and $V_\varepsilon$ are of the form \eqref{E:MTD_UON} and \eqref{E:MTD_ONV}, we have
\begin{align} \label{E:PM_Psee}
  \begin{aligned}
    \psi_\eta^\varepsilon &= \nu_\varepsilon\cdot\nabla\rho_\eta^\varepsilon+k_d^{-1}V_\varepsilon\rho_\eta^\varepsilon \\
    &= \frac{(-1)^{i+1}}{\sqrt{1+\varepsilon^2|\bar{\tau}_\varepsilon^i|^2}}\left\{(\bar{\nu}-\varepsilon\bar{\tau}_\varepsilon^i)\cdot\nabla\rho_\eta^\varepsilon+k_d^{-1}\Bigl(\overline{V_\Gamma}+\varepsilon\,\overline{\partial^\circ g_i}+\varepsilon^2\bar{g}_i\bar{\tau}_\varepsilon^i\cdot\overline{\nabla_\Gamma V_\Gamma}\Bigr)\rho_\eta^\varepsilon\right\}
  \end{aligned}
\end{align}
on $\Gamma_\varepsilon^i(t)$, where $\tau_\varepsilon^i$ is given by \eqref{E:TD_AuxVec}.
To estimate the second line, we first note that
\begin{align*}
  \bar{\nu}\cdot\nabla d = |\bar{\nu}|^2 = 1, \quad \bar{\nu}\cdot\nabla\bar{\xi} = 0 \quad\text{in}\quad \overline{N_T}
\end{align*}
for a function $\xi$ on $\overline{S_T}$.
We apply these equalities to $\rho_\eta^\varepsilon$ of the form \eqref{E:PM_ee12} to get
\begin{align*}
  \bar{\nu}\cdot\nabla\rho_\eta^\varepsilon = -k_d^{-1}\Bigl(\overline{V_\Gamma\eta}\Bigr)+\varepsilon\Bigl(\bar{g}_1\bar{\zeta}_0-\bar{g}_0\bar{\zeta}_1\Bigr)+d\Bigl(\bar{\zeta}_1-\bar{\zeta}_0\Bigr) \quad\text{in}\quad \overline{Q_{\varepsilon,T}}.
\end{align*}
Then we see by $d=\varepsilon\bar{g}_i$ on $\Gamma_\varepsilon^i(t)$, $g_1-g_0=g$ on $S_T$, and \eqref{E:PM_zeta} that
\begin{align} \label{E:PM_NuGe}
  \begin{aligned}
    \bar{\nu}\cdot\nabla\rho_\eta^\varepsilon &= -k_d^{-1}\Bigl(\overline{V_\Gamma\eta}\Bigr)+\varepsilon\Bigl(\overline{g\zeta_i}\Bigr) \\
    &= -k_d^{-1}\Bigl(\overline{V_\Gamma\eta}\Bigr)+\varepsilon\left\{\overline{\nabla_\Gamma g_i}\cdot\overline{\nabla_\Gamma\eta}-k_d^{-1}\Bigl(\overline{\partial^\circ g_i}\Bigr)\bar{\eta}+k_d^{-2}\Bigl(\overline{g_iV_\Gamma^2\eta}\Bigr)\right\}
  \end{aligned}
\end{align}
on $\Gamma_\varepsilon^i(t)$.
Moreover, noting that $\bar{\tau}_\varepsilon^i\cdot\nabla d=\bar{\tau}_\varepsilon^i\cdot\bar{\nu}=0$ in $\overline{N_T}$ by \eqref{E:TD_AVTan}, we deduce from \eqref{E:CEGr_Bound} and \eqref{E:CEGr_Diff} with $d=\varepsilon\bar{g}_i$ on $\Gamma_\varepsilon^i(t)$ and \eqref{E:TD_AVBo} that
\begin{align} \label{E:PM_TauGe}
  \Bigl|\bar{\tau}_\varepsilon^i\cdot\nabla\rho_\eta^\varepsilon-\overline{\nabla_\Gamma g_i}\cdot\overline{\nabla_\Gamma\eta}\Bigr| \leq c\varepsilon\sum_{\xi=\eta,\zeta_0,\zeta_1}\left(|\bar{\xi}|+\Bigl|\overline{\nabla_\Gamma\xi}\Bigr|\right)
\end{align}
on $\Gamma_\varepsilon^i(t)$.
We also observe by $d=\varepsilon\bar{g}_i$ on $\Gamma_\varepsilon^i(t)$ that
\begin{align} \label{E:PM_BoZero}
  \left|\Bigl(\overline{V_\Gamma}+\varepsilon\,\overline{\partial^\circ g_i}\Bigr)\rho_\eta^\varepsilon-\overline{V_\Gamma\eta}-\varepsilon\left\{\Bigl(\overline{\partial^\circ g_i}\Bigr)\bar{\eta}-k_d^{-1}\Bigl(\overline{g_iV_\Gamma^2\eta}\Bigr)\right\}\right| \leq c\varepsilon^2(|\bar{\eta}|+|\bar{\zeta}_0|+|\bar{\zeta}_1|)
\end{align}
on $\Gamma_\varepsilon^i(t)$.
Noting that $|\rho_\eta^\varepsilon|\leq c(|\bar{\eta}|+|\bar{\zeta}_0|+|\bar{\zeta}_1|)$ in $\overline{Q_{\varepsilon,T}}$, we apply \eqref{E:TD_AVBo} and \eqref{E:PM_NuGe}--\eqref{E:PM_BoZero} to \eqref{E:PM_Psee} to find that
\begin{align*}
  |\psi_\eta^\varepsilon| &\leq \left|(\bar{\nu}-\varepsilon\bar{\tau}_\varepsilon^i)\cdot\nabla\rho_\eta^\varepsilon+k_d^{-1}\Bigl(\overline{V_\Gamma}+\varepsilon\,\overline{\partial^\circ g_i}+\varepsilon^2\bar{g}_i\bar{\tau}_\varepsilon^i\cdot\overline{\nabla_\Gamma V_\Gamma}\Bigr)\rho_\eta^\varepsilon\right| \\
  &\leq c\varepsilon^2\sum_{\xi=\eta,\zeta_0,\zeta_1}\left(|\bar{\xi}|+\Bigl|\overline{\nabla_\Gamma\xi}\Bigr|\right)
\end{align*}
on $\Gamma_\varepsilon^i(t)$ with $i=0,1$ and $t\in(0,T]$.
Therefore,
\begin{align} \label{E:PM_PseB}
  \|\psi_\eta^\varepsilon\|_{C(\partial_\ell Q_{\varepsilon,T})} \leq c\varepsilon^2\Bigl(\|\eta\|_{C^{2,1}(S_T)}+\|\zeta_0\|_{C^{2,1}(S_T)}+\|\zeta_1\|_{C^{2,1}(S_T)}\Bigr).
\end{align}
Finally, we obtain \eqref{E:Error_Gen} by applying \eqref{E:PM_e12B}, \eqref{E:PM_feeB}, and \eqref{E:PM_PseB} to \eqref{E:PM_Diff} and noting that
\begin{align*}
  \|\eta\|_{C(\overline{S_T})} \leq \|\eta_0\|_{C(\Gamma(0))}+\|\eta\|_{C^{2,1}(S_T)}
\end{align*}
by $\eta|_{t=0}=\eta_0$ on $\Gamma(0)$ and that
\begin{align*}
  \|\zeta_i\|_{C(\overline{S_T})}+\|\zeta_i\|_{C^{2,1}(S_T)} &\leq \|\zeta_i(\cdot,0)\|_{C(\Gamma(0))}+\|\zeta_i\|_{C(S_T)}+\|\zeta_i\|_{C^{2,1}(S_T)} \\
  &\leq c\Bigl(\|\eta_0\|_{C(\Gamma(0))}+\|\eta\|_{C^{2,1}(S_T)}\Bigr) \\
  &\qquad +c\Bigl(\|h_i(\cdot,0)\|_{C(\Gamma(0))}+\|h_i\|_{C^{2,1}(S_T)}\Bigr)
\end{align*}
for $i=0,1$ with $h_i=\nabla_\Gamma g_i\cdot\nabla_\Gamma\eta$ on $\overline{S_T}$ by \eqref{E:PM_zeta}.
The proof of Theorem \ref{T:Error_Gen} is complete.

\begin{remark} \label{R:PM_UapE}
  By the idea of the above proof, we can also show that the factor $\varepsilon^{-1}$ in the uniform a priori estimate \eqref{E:Uni_Heat} cannot be removed.
  Indeed, assume that the inequality
  \begin{align} \label{E:Wrong_Uni}
    \|\rho^\varepsilon\|_{C(\overline{Q_{\varepsilon,T}})} \leq c\left(\|\rho_0^\varepsilon\|_{C(\overline{\Omega_\varepsilon(0)})}+\|f^\varepsilon\|_{C(Q_{\varepsilon,T})}+\|\psi^\varepsilon\|_{C(\partial_\ell Q_{\varepsilon,T})}\right)
  \end{align}
  holds for a classical solution $\rho^\varepsilon$ to \eqref{E:Heat_Gen} with a constant $c>0$ independent of $\varepsilon$.
  For an arbitrary smooth function $\zeta$ on $\overline{S_T}$, we define
  \begin{align*}
    \rho_\zeta^\varepsilon(x,t) = \bar{\zeta}(x,t)-k_d^{-1}d(x,t)\Bigl(\overline{V_\Gamma\zeta}\Bigr)(x,t)+\frac{1}{2}d(x,t)^2\bar{\zeta}_2(x,t), \quad (x,t)\in\overline{Q_{\varepsilon,T}},
  \end{align*}
  where $\zeta_2$ is a function on $\overline{S_T}$ given later (and $\rho_\zeta^\varepsilon$ is in fact independent of $\varepsilon$), and set
  \begin{align*}
    f_\zeta^\varepsilon &= \partial_t\rho_\zeta^\varepsilon-k_d^{-1}\Delta\rho_\zeta^\varepsilon \quad\text{on}\quad Q_{\varepsilon,T}, \quad \psi_\zeta^\varepsilon = \partial_{\nu_\varepsilon}\rho_\zeta^\varepsilon+k_d^{-1}V_\varepsilon\rho_\zeta^\varepsilon \quad\text{on}\quad \partial_\ell Q_{\varepsilon,T}.
  \end{align*}
  Then since $\rho_\zeta^\varepsilon$ satisfies \eqref{E:Heat_Gen} with the above $f_\zeta^\varepsilon$ and $\psi_\zeta^\varepsilon$, we can use \eqref{E:Wrong_Uni} to get
  \begin{align*}
    \|\rho_\zeta^\varepsilon\|_{C(\overline{Q_{\varepsilon,T}})} \leq c\left(\|\rho_\zeta^\varepsilon(\cdot,0)\|_{C(\overline{\Omega_\varepsilon(0)})}+\|f_\zeta^\varepsilon\|_{C(Q_{\varepsilon,T})}+\|\psi_\zeta^\varepsilon\|_{C(\partial_\ell Q_{\varepsilon,T})}\right)
  \end{align*}
  and thus, by the definition of $\rho_\zeta^\varepsilon$, $|d|\leq c\varepsilon$ in $\overline{Q_{\varepsilon,T}}$, and
  \begin{align*}
    \|\bar{\zeta}\|_{C(\overline{Q_{\varepsilon,T}})} = \|\zeta\|_{C(\overline{S_T})}, \quad \|\bar{\zeta}(\cdot,0)\|_{C(\overline{\Omega_\varepsilon(0)})} = \|\zeta(\cdot,0)\|_{C(\Gamma(0))} \leq \|\zeta\|_{C(\overline{S_T})},
  \end{align*}
  we find that
  \begin{align} \label{E:Wr_Zeta}
    \begin{aligned}
      \|\zeta\|_{C(\overline{S_T})} &\leq c\Bigl(\|\zeta(\cdot,0)\|_{C(\Gamma(0))}+\|f_\zeta^\varepsilon\|_{C(Q_{\varepsilon,T})}+\|\psi_\zeta^\varepsilon\|_{C(\partial_\ell Q_{\varepsilon,T})}\Bigr) \\
      &\qquad +c\varepsilon\Bigl(\|\zeta\|_{C(\overline{S_T})}+\|\zeta_2\|_{C(\overline{S_T})}\Bigr).
    \end{aligned}
  \end{align}
  Moreover, since
  \begin{align*}
    \bar{\nu}(x,t)\cdot\nabla\rho_\zeta^\varepsilon(x,t) = -k_d^{-1}\Bigl(\overline{V_\Gamma\zeta}\Bigr)(x,t)+d(x,t)\bar{\zeta}_2(x,t), \quad (x,t)\in\overline{Q_{\varepsilon,T}}
  \end{align*}
  by $\nabla d=\bar{\nu}$, $|\bar{\nu}|^2=1$, and $\bar{\nu}\cdot\nabla\bar{\xi}=0$ in $\overline{N_T}$ for a function $\xi$ on $\overline{S_T}$, we observe by \eqref{E:PM_Psee} with $\eta$ replaced by $\zeta$ and by \eqref{E:TD_AVBo} and $d=\varepsilon\bar{g}_i$ on $\Gamma_\varepsilon^i(t)$ that
  \begin{align} \label{E:Wr_Pz_Bd}
    |\psi_\zeta^\varepsilon| \leq c\varepsilon\sum_{\xi=\zeta,\zeta_2}\left(|\bar{\xi}|+\Bigl|\overline{\nabla_\Gamma\xi}\Bigr|\right) \quad\text{on}\quad \partial_\ell Q_{\varepsilon,T}.
  \end{align}
  Also, as in \eqref{E:PM_dtAp} and \eqref{E:PM_LapAp}, we have
  \begin{align} \label{E:Wr_Fz_Bd}
    |f_\zeta^\varepsilon-\bar{f}_\zeta| \leq c\varepsilon\sum_{\xi=\zeta,\zeta_2}\left(|\bar{\xi}|+\Bigl|\overline{\partial^\circ\xi}\Bigr|+\Bigl|\overline{\nabla_\Gamma\xi}\Bigr|+\Bigl|\overline{\nabla_\Gamma^2\xi}\Bigr|\right) \quad\text{in}\quad Q_{\varepsilon,T}
  \end{align}
  by \eqref{E:CE_Lap}--\eqref{E:CEL_2d}, \eqref{E:SDdt_Surf}, \eqref{E:CEdt_Bound}, \eqref{E:CEdt_Diff}, and $|d|\leq c\varepsilon$ in $\overline{Q_{\varepsilon,T}}$, where
  \begin{align*}
    f_\zeta = \partial^\circ\zeta+k_d^{-1}V_\Gamma^2\zeta-k_d\Delta_\Gamma\zeta-V_\Gamma H\zeta-k_d\zeta_2 \quad\text{on}\quad S_T.
  \end{align*}
  Thus, setting
  \begin{align*}
    \zeta_2 = k_d^{-1}(\partial^\circ\zeta+k_d^{-1}V_\Gamma^2\zeta-k_d\Delta_\Gamma\zeta-V_\Gamma H\zeta) \quad\text{on}\quad \overline{S_T},
  \end{align*}
  we get $f_\zeta=0$ on $S_T$, and it follows from \eqref{E:Wr_Pz_Bd}, \eqref{E:Wr_Fz_Bd}, and $f_\zeta=0$ on $S_T$ that
  \begin{align*}
    \|f_\zeta^\varepsilon\|_{C(Q_{\varepsilon,T})}+\|\psi_\zeta^\varepsilon\|_{C(\partial_\ell Q_{\varepsilon,T})} \leq c\varepsilon\Bigl(\|\zeta\|_{C^{2,1}(S_T)}+\|\zeta_2\|_{C^{2,1}(S_T)}\Bigr).
  \end{align*}
  Now we apply this inequality to \eqref{E:Wr_Zeta} to find that
  \begin{align*}
    \|\zeta\|_{C(\overline{S_T})} \leq c\|\zeta(\cdot,0)\|_{C(\Gamma(0))}+c\varepsilon\sum_{\xi=\zeta,\zeta_2}\Bigl(\|\xi\|_{C(\overline{S_T})}+\|\xi\|_{C^{2,1}(S_T)}\Bigr)
  \end{align*}
  with a constant $c>0$ independent of $\varepsilon$.
  Hence we send $\varepsilon\to0$ to get
  \begin{align*}
    \|\zeta\|_{C(\overline{S_T})} \leq c\|\zeta(\cdot,0)\|_{C(\Gamma(0))}
  \end{align*}
  for any smooth function $\zeta$ on $\overline{S_T}$, but this is a contradiction since we can take $\zeta$ such that $\zeta(\cdot,0)=0$ on $\Gamma(0)$ but $\zeta(\cdot,t)$ does not vanish on $\Gamma(t)$ for some $t\in(0,T]$.
  Therefore, the inequality \eqref{E:Wrong_Uni} does not hold, i.e. we cannot remove $\varepsilon^{-1}$ in \eqref{E:Uni_Heat}.
  In the same way, we can show that any smooth function on $\overline{S_T}$ approximates a classical solution to \eqref{E:Heat_MTD} of order $\varepsilon$ as in \eqref{E:Error_Spe} and \eqref{E:Error_Gen} if \eqref{E:Wrong_Uni} holds, which is again absurd.
  Hence we can also consider that the factor $\varepsilon^{-1}$ in \eqref{E:Uni_Heat} is crucial in order that $\eta$ should be a solution to the limit equation \eqref{E:Limit} in the error estimates \eqref{E:Error_Spe} and \eqref{E:Error_Gen}.
\end{remark}

\section{Formal derivation of a limit equation and an approximate solution} \label{S:Formal}
We explain how to derive formally the limit equation \eqref{E:Limit} and the approximate solution $\rho_\eta^\varepsilon$ of the form \eqref{E:PM_ee12} from the thin domain problem \eqref{E:Heat_MTD}.
Throughout this section, we write $O(\varepsilon^k)$ with $k\geq1$ for any quantity of order at least $\varepsilon^k$.

First we assume that a function $\rho$ on $\overline{Q_{\varepsilon,T}}$ is of the form
\begin{align*}
  \rho(x,t) = \eta(\pi(x,t),t,\varepsilon^{-1}d(x,t)), \quad (x,t) \in \overline{Q_{\varepsilon,T}},
\end{align*}
where $\eta(y,t,r)$ is a function of $(y,t)\in\overline{S_T}$ and $r\in[g_0(y,t),g_1(y,t)]$ independent of $\varepsilon$.
By a slight abuse of the notation, we write
\begin{align*}
  \bar{\eta}(x,t,r) = \eta(\pi(x,t),t,r), \quad (x,t) \in \overline{N_T}, \, r\in[\bar{g}_0(x,t),\bar{g}_1(x,t)]
\end{align*}
so that $\rho(x,t)=\bar{\eta}(x,t,\varepsilon^{-1}d)$ with $d=d(x,t)$.
In what follows, we sometimes suppress the arguments $x$ and $t$ of functions which do not have the third argument $r$.
Also, we consider that the tangential derivatives apply to the argument $y$ of $\eta(y,t,r)$.

Let $(x,t)\in\overline{Q_{\varepsilon,T}}$ and $R$ be given by \eqref{E:Inv_IdW}.
We see by \eqref{E:SDdt_TN} and \eqref{E:CEdt_TN} that
\begin{align*}
  \partial_t\rho &= \partial_t\bar{\eta}(x,t,\varepsilon^{-1}d)+\varepsilon^{-1}(\partial_td)\partial_r\bar{\eta}(x,t,\varepsilon^{-1}d) \\
  &= \overline{\partial^\circ\eta}(x,t,\varepsilon^{-1}d)+d\Bigl(R\,\overline{\nabla_\Gamma V_\Gamma}\Bigr)\cdot\overline{\nabla_\Gamma\eta}(x,t,\varepsilon^{-1}d)-\varepsilon^{-1}\overline{V_\Gamma}\,\partial_r\bar{\eta}(x,t,\varepsilon^{-1}d).
\end{align*}
Then, setting $y=\pi(x,t)$ and $r=\varepsilon^{-1}d(x,t)$ and noting that
\begin{align} \label{E:FD_R}
  R(x,t) = \{I_n-\varepsilon rW(y,t)\}^{-1} = I_n+\varepsilon rW(y,t)+O(\varepsilon^2)
\end{align}
by the Neumann series expression, we have
\begin{align} \label{E:FD_dt}
  \partial_t\rho(x,t) = -\varepsilon^{-1}V_\Gamma(y,t)\partial_r\eta(y,t,r)+\partial^\circ\eta(y,t,r)+O(\varepsilon).
\end{align}
Next we observe by \eqref{E:CEGr_TN} and $\nabla d=\bar{\nu}$ in $\overline{N_T}$ that
\begin{align} \label{E:FD_par_i}
  \partial_i\rho = \sum_{j=1}^nR_{ij}\overline{\underline{D}_j\eta}(x,t,\varepsilon^{-1}d)+\varepsilon^{-1}\bar{\nu}_i\partial_r\bar{\eta}(x,t,\varepsilon^{-1}d)
\end{align}
for $i=1,\dots,n$.
Then we use \eqref{E:FD_R} with $y=\pi(x,t)$ and $r=\varepsilon^{-1}d(x,t)$ to find that
\begin{align} \label{E:FD_Gr}
  \nabla\rho(x,t) &= \varepsilon^{-1}\partial_r\eta(y,t,r)\nu(y,t)+\nabla_\Gamma\eta(y,t,r)+\varepsilon rW(y,t)\nabla_\Gamma\eta(y,t,r)+O(\varepsilon^2).
\end{align}
Also, we apply $\partial_i$ to \eqref{E:FD_par_i} and use \eqref{E:CEGr_TN} and $\nabla d=\bar{\nu}$ in $\overline{N_T}$ to get
\begin{align*}
  \partial_i\partial_i\rho &= \sum_{j=1}^n(\partial_iR_{ij})\overline{\underline{D}_j\eta}(x,t,\varepsilon^{-1}d)+\sum_{j,k=1}^nR_{ij}R_{ik}\overline{\underline{D}_k\underline{D}_j\eta}(x,t,\varepsilon^{-1}d) \\
  &\qquad +\sum_{j=1}^nR_{ij}\left\{\varepsilon^{-1}\bar{\nu}_i\partial_r\Bigl(\overline{\underline{D}_j\eta}\Bigr)(x,t,\varepsilon^{-1}d)\right\} \\
  &\qquad +\varepsilon^{-1}\sum_{j=1}^nR_{ij}\Bigl(\overline{\underline{D}_j\nu_i}\Bigr)\partial_r\bar{\eta}(x,t,\varepsilon^{-1}d)\\
  &\qquad +\varepsilon^{-1}\bar{\nu}_i\sum_{j=1}^nR_{ij}\overline{\underline{D}_j(\partial_r\eta)}(x,t,\varepsilon^{-1}d)+\varepsilon^{-2}\bar{\nu}_i^2\partial_r^2\bar{\eta}(x,t,\varepsilon^{-1}d).
\end{align*}
We take the sum of both sides for $i=1,\dots,n$.
Then, since
\begin{align*}
  R^T\bar{\nu} = R\bar{\nu} = \Bigl\{I_n-d\overline{W}\Bigr\}^{-1}\bar{\nu} = \bar{\nu}, \quad\text{i.e.}\quad \sum_{i=1}^nR_{ij}\bar{\nu}_i = \bar{\nu}_j \quad\text{in}\quad \overline{N_T}
\end{align*}
for $j=1,\dots,n$ by the symmetry of $W$ and $W\nu=0$ on $\overline{S_T}$, since
\begin{align*}
  R_{ij}(x,t) = \delta_{ij}+\varepsilon rW_{ij}(y,t)+O(\varepsilon^2), \quad \partial_iR_{ij}(x,t) = (\nu_iW_{ij})(y,t)+O(\varepsilon)
\end{align*}
for $(x,t)\in\overline{Q_{\varepsilon,T}}$ and $i,j=1,\dots,n$ by \eqref{E:IdW_Der} with $|d|\leq c\varepsilon$ in $\overline{Q_{\varepsilon,T}}$ and by \eqref{E:FD_R}, where $\delta_{ij}$ is the Kronecker delta and $y=\pi(x,t)$ and $r=\varepsilon^{-1}d(x,t)$, and since
\begin{align*}
  |\nu|^2 = 1, \quad \underline{D}_j\nu_i = -W_{ji}, \quad \sum_{i=1}^n\underline{D}_i\nu_i = \mathrm{div}_\Gamma\nu = -H \quad\text{on}\quad \overline{S_T},
\end{align*}
we have
\begin{align*}
  \Delta\rho(x,t) &= \sum_{i,j=1}^n(\nu_iW_{ij})(y,t)\underline{D}_j\eta(y,t,r)+\Delta_\Gamma\eta(y,t,r) \\
  &\qquad +\varepsilon^{-1}\sum_{j=1}^n\nu_j(y,t)\partial_r(\underline{D}_j\eta)(y,t,r)\\
  &\qquad -\varepsilon^{-1}H(y,t)\partial_r\eta(y,t,r)-r\sum_{i,j=1}^n(W_{ij}W_{ji})(y,t)\partial_r\eta(y,t,r) \\
  &\qquad +\varepsilon^{-1}\sum_{j=1}^n\nu_j(y,t)\underline{D}_j(\partial_r\eta)(y,t,r)+\varepsilon^{-2}\partial_r^2\eta(y,t,r)+O(\varepsilon)
\end{align*}
with $y=\pi(y,t)$ and $r=\varepsilon^{-1}d(x,t)$.
Moreover, we see that
\begin{align*}
  \sum_{i,j=1}^n(\nu_iW_{ij})(y,t)\underline{D}_j\eta(y,t,r) &= (W^T\nu)(y,t)\cdot\nabla_\Gamma\eta(y,t,r) = 0, \\
  \sum_{i,j=1}^n(W_{ij}W_{ji})(y,t) &= \sum_{i,j=1}^nW_{ij}(y,t)^2 = |W(y,t)|^2
\end{align*}
by $W^T=W$ and $W\nu=0$ on $\overline{S_T}$, and that
\begin{align*}
  \sum_{j=1}^n\nu_j(y,t)\partial_r(\underline{D}_j\eta)(y,t,r) &= \partial_r\Bigl(\nu(y,t)\cdot\nabla_\Gamma\eta(y,t,r)\Bigr) = 0, \\
  \sum_{j=1}^n\nu_j(y,t)\underline{D}_j(\partial_r\eta)(y,t,r) &= \nu(y,t)\cdot[\nabla_\Gamma(\partial_r\eta)](y,t,r) = 0,
\end{align*}
since $\nu$ is independent of the variable $r$ and $\nu\cdot\nabla_\Gamma=0$.
Therefore,
\begin{align} \label{E:FD_Lap}
  \begin{aligned}
    \Delta\rho(x,t) &= \varepsilon^{-2}\partial_r^2\eta(y,t,r)-\varepsilon^{-1}H(y,t)\partial_r\eta(y,t,r) \\
    &\qquad -r|W(y,t)|^2\partial_r\eta(y,t,r)+\Delta_\Gamma\eta(y,t,r)+O(\varepsilon).
  \end{aligned}
\end{align}
Recall that we take $(x,t)\in\overline{Q_{\varepsilon,T}}$ and set $y=\pi(x,t)$ and $r=\varepsilon^{-1}d(x,t)$.

Now let $\rho^\varepsilon$ be a solution to the thin domain problem \eqref{E:Heat_MTD}.
To derive the limit equation \eqref{E:Limit} and the approximate solution $\rho_\eta^\varepsilon$ of the form \eqref{E:PM_ee12}, we consider an asymptotic expansion of $\rho^\varepsilon$ as in the case of a flat stationary thin domain (see e.g. \cite[Section 15.1]{MaNaPl00_02}) but in a slightly simplified form: we assume that $\rho^\varepsilon$ and $f^\varepsilon$ are of the form
\begin{align} \label{E:FD_Asymp}
  \begin{aligned}
    \rho^\varepsilon(x,t) &= \sum_{k=0}^\infty\varepsilon^k\eta_k(\pi(x,t),t,\varepsilon^{-1}d(x,t)), \\
    f^\varepsilon(x,t) &= f(\pi(x,t),t)+O(\varepsilon),
  \end{aligned}
\end{align}
where the functions $\eta_k(y,t,r)$ are independent of $\varepsilon$ and $f(y,t)$ is a function of $(y,t)\in S_T$ independent of $\varepsilon$.
Our aim is to give $\eta_0$, $\eta_1$, and $\eta_2$ for which the right-hand side of \eqref{E:FD_Asymp} satisfies \eqref{E:Heat_MTD} approximately of order $\varepsilon$ in $Q_{\varepsilon,T}$ and $\varepsilon^2$ on $\partial_\ell Q_{\varepsilon,T}$.
In what follows, we set $y=\pi(x,t)$ for $(x,t)\in\overline{Q_{\varepsilon,T}}$ and suppress the arguments $y$ and $t$ of $\eta_k(y,t,r)$ and functions on $\overline{S_T}$.

First we consider the heat equation in $Q_{\varepsilon,T}$.
We see by \eqref{E:FD_dt} that
\begin{align*}
  \partial_t\rho^\varepsilon(x,t) = -\varepsilon^{-1}V_\Gamma\partial_r\eta_0(r)+\partial^\circ\eta_0(r)-V_\Gamma\partial_r\eta_1(r)+O(\varepsilon)
\end{align*}
for $(x,t)\in Q_{\varepsilon,T}$ and $r=\varepsilon^{-1}d(x,t)\in(g_0,g_1)$.
Also, by \eqref{E:FD_Lap},
\begin{align*}
  \Delta\rho^\varepsilon(x,t) &= \varepsilon^{-2}\partial_r^2\eta_0(r)+\varepsilon^{-1}\{-H\partial_r\eta_0(r)+\partial_r^2\eta_1(r)\} \\
  &\qquad -r|W|^2\partial_r\eta_0(r)+\Delta_\Gamma\eta_0(r)-H\partial_r\eta_1(r)+\partial_r^2\eta_2(r)+O(\varepsilon).
\end{align*}
We substitute these expressions for
\begin{align*}
  \partial_t\rho^\varepsilon(x,t)-k_d\Delta\rho^\varepsilon(x,t) = f^\varepsilon(x,t) = f+O(\varepsilon), \quad (x,t)\in Q_{\varepsilon,T}.
\end{align*}
Then, for $r\in(g_0,g_1)$, we have
\begin{align*}
  &\varepsilon^{-2}\{-k_d\partial_r^2\eta_0(r)\}+\varepsilon^{-1}\{-V_\Gamma\partial_r\eta_0(r)+k_dH\partial_r\eta_0(r)-k_d\partial_r^2\eta_1(r)\} \\
  &\qquad +\partial^\circ\eta_0(r)-V_\Gamma\partial_r\eta_1(r)+k_dr|W|^2\partial_r\eta_0(r) \\
  &\qquad -k_d\Delta_\Gamma\eta_0(r)+k_dH\partial_r\eta_1(r)-k_d\partial_r^2\eta_2(r) = f+O(\varepsilon).
\end{align*}
Since each function in this equation is independent of $\varepsilon$, it follows that
\begin{align}
  -k_d\partial_r^2\eta_0(r) &= 0, \label{E:FD_e0_In} \\
  -V_\Gamma\partial_r\eta_0(r)+k_dH\partial_r\eta_0(r)-k_d\partial_r^2\eta_1(r) &= 0, \label{E:FD_e1_In}
\end{align}
and
\begin{align} \label{E:FD_e2_In}
  \begin{aligned}
    &\partial^\circ\eta_0(r)-V_\Gamma\partial_r\eta_1(r)+k_dr|W|^2\partial_r\eta_0(r) \\
    &\qquad -k_d\Delta_\Gamma\eta_0(r)+k_dH\partial_r\eta_1(r)-k_d\partial_r^2\eta_2(r) = f
  \end{aligned}
\end{align}
for $r\in(g_0,g_1)$.

Next we deal with the boundary condition on $\partial_\ell Q_{\varepsilon,T}$.
For $t\in(0,T]$ and $i=0,1$, let $x\in\Gamma_\varepsilon^i(t)$ and $r=\varepsilon^{-1}d(x,t)=g_i$.
Since $\nu_\varepsilon$ and $V_\varepsilon$ are of the form \eqref{E:MTD_UON} and \eqref{E:MTD_ONV}, the boundary condition is equivalent to
\begin{align} \label{E:FD_Bdry}
  \bigl[(\bar{\nu}-\varepsilon\bar{\tau}_\varepsilon^i)\cdot\nabla\rho^\varepsilon\bigr](x,t)+k_d^{-1}\left[\Bigl(\overline{V_\Gamma}+\varepsilon\,\overline{\partial^\circ g_i}+\varepsilon^2\bar{g}_i\bar{\tau}_\varepsilon^i\cdot\overline{\nabla_\Gamma V_\Gamma}\Bigr)\rho^\varepsilon\right](x,t) = 0.
\end{align}
Here $\tau_\varepsilon^i$ is given by \eqref{E:TD_AuxVec} and thus
\begin{align*}
  \bar{\tau}_\varepsilon^i(x,t) = (I_n-\varepsilon g_iW)^{-1}\nabla_\Gamma g_i = \nabla_\Gamma g_i+\varepsilon g_iW\nabla_\Gamma g_i+O(\varepsilon^2)
\end{align*}
by the Neumann series expression.
Also, we see by \eqref{E:FD_Gr} that
\begin{align*}
  \nabla\rho^\varepsilon(x,t) &= \varepsilon^{-1}\partial_r\eta_0(r)\nu+\nabla_\Gamma\eta_0(r)+\partial_r\eta_1(r)\nu \\
  &\qquad +\varepsilon\{rW\nabla_\Gamma\eta_0(r)+\nabla_\Gamma\eta_1(r)+\partial_r\eta_2(r)\nu\}+O(\varepsilon^2).
\end{align*}
By the above expressions, $|\nu|^2=1$, $W^T\nu=W\nu=0$, and $\nu\cdot\nabla_\Gamma=0$, we get
\begin{align*}
  \bigl[(\bar{\nu}-\varepsilon\bar{\tau}_\varepsilon^i)\cdot\nabla\rho^\varepsilon\bigr](x,t) &= \varepsilon^{-1}\partial_r\eta_0(r)+\partial_r\eta_1(r) \\
  &\qquad +\varepsilon\{\partial_r\eta_2(r)-\nabla_\Gamma g_i\cdot\nabla_\Gamma\eta_0(r)\}+O(\varepsilon^2).
\end{align*}
Note that $\bar{\tau}_\varepsilon^i$ in the left-hand side is multiplied by $\varepsilon$.
We also have
\begin{align*}
  &\left[\Bigl(\overline{V_\Gamma}+\varepsilon\,\overline{\partial^\circ g_i}+\varepsilon^2\bar{g}_i\bar{\tau}_\varepsilon^i\cdot\overline{\nabla_\Gamma V_\Gamma}\Bigr)\rho^\varepsilon\right](x,t) \\
  &\qquad = V_\Gamma\eta_0(r)+\varepsilon\{(\partial^\circ g_i)\eta_0(r)+V_\Gamma\eta_1(r)\}+O(\varepsilon^2),
\end{align*}
and we substitute the above expressions for \eqref{E:FD_Bdry} to obtain
\begin{align*}
  &\varepsilon^{-1}\partial_r\eta_0(r)+\partial_r\eta_1(r)+k_d^{-1}V_\Gamma\eta_0(r) \\
  &\qquad +\varepsilon\{\partial_r\eta_2(r)-\nabla_\Gamma g_i\cdot\nabla_\Gamma\eta_0(r)+k_d^{-1}(\partial^\circ g_i)\eta_0(r)+k_d^{-1}V_\Gamma\eta_1(r)\} = O(\varepsilon^2).
\end{align*}
Since each function in the left-hand side is independent of $\varepsilon$, we find that
\begin{align}
  \partial_r\eta_0(r) &= 0, \label{E:FD_e0_Bo} \\
  \partial_r\eta_1(r)+k_d^{-1}V_\Gamma\eta_0(r) &= 0, \label{E:FD_e1_Bo}
\end{align}
and
\begin{align} \label{E:FD_e2_Bo}
  \partial_r\eta_2(r)-\nabla_\Gamma g_i\cdot\nabla_\Gamma\eta_0(r)+k_d^{-1}(\partial^\circ g_i)\eta_0(r)+k_d^{-1}V_\Gamma\eta_1(r) = 0
\end{align}
for $r=g_i$ with $i=0,1$ (recall that we suppress the arguments $y$ and $t$).

Now let us determine $\eta_0$, $\eta_1$, and $\eta_2$.
By \eqref{E:FD_e0_In} and \eqref{E:FD_e0_Bo}, we see that
\begin{align} \label{E:FD_e0_Eq}
  \partial_r^2\eta_0(r) = 0, \quad r\in(g_0,g_1), \quad \partial_r\eta_0(g_i) = 0, \quad i=0,1.
\end{align}
Hence $\partial_r\eta_0(r)=0$ for $r\in[g_0,g_1]$ and $\eta_0(r)=\eta_0$ is independent of $r$.
Note that at this point we do not know what equation $\eta_0$ should satisfy on $S_T$.
Next we have
\begin{align*}
  \partial_r^2\eta_1(r) = 0, \quad r\in(g_0,g_1), \quad \partial_r\eta_1(g_i) = -k_d^{-1}V_\Gamma\eta_0, \quad i=0,1
\end{align*}
by \eqref{E:FD_e1_In}, \eqref{E:FD_e1_Bo}, and $\eta_0(r)=\eta_0$.
Hence
\begin{align*}
  \partial_r\eta_1(r) = -k_d^{-1}V_\Gamma\eta_0, \quad \eta_1(r) = \eta_1(g_0)-k_d^{-1}(r-g_0)V_\Gamma\eta_0, \quad r\in[g_0,g_1].
\end{align*}
Here we may choose $\eta_1(g_0)$ arbitrarily since we only consider the approximation of \eqref{E:Heat_MTD} of order $\varepsilon$ in $Q_{\varepsilon,T}$ and $\varepsilon^2$ on $\partial_\ell Q_{\varepsilon,T}$.
We take $\eta_1(g_0)=-k_d^{-1}g_0V_\Gamma\eta_0$ so that
\begin{align} \label{E:FD_eta1}
  \eta_1(r) = -k_d^{-1}rV_\Gamma\eta_0, \quad r\in[g_0,g_1].
\end{align}
However, if we require a higher order approximation of \eqref{E:Heat_MTD}, then $\eta_1(g_0)$ should be determined by a higher order asymptotic expansion of \eqref{E:Heat_MTD}.
For $\eta_2$, we have
\begin{align} \label{E:FD_e2_Eq}
  \left\{
  \begin{aligned}
    \partial_r^2\eta_2(r) &= k_d^{-1}\partial^\circ\eta_0+k_d^{-2}V_\Gamma^2\eta_0-\Delta_\Gamma\eta_0-k_d^{-1}V_\Gamma H\eta_0-k_d^{-1}f, \quad r\in(g_0,g_1), \\
    \partial_r\eta_2(g_i) &= \nabla_\Gamma g_i\cdot\nabla_\Gamma\eta_0-k_d^{-1}(\partial^\circ g_i)\eta_0+k_d^{-2}g_iV_\Gamma^2\eta_0, \quad i=0,1
  \end{aligned}
  \right.
\end{align}
by \eqref{E:FD_e2_In}, \eqref{E:FD_e2_Bo}, $\eta_0(r)=\eta_0$, and \eqref{E:FD_eta1}.
In order that $\eta_2$ exists, it is necessary that
\begin{align*}
  \int_{g_0}^{g_1}\partial_r^2\eta_2(r)\,dr = \partial_r\eta_2(g_1)-\partial_r\eta_2(g_0).
\end{align*}
Then, noting that the right-hand side of the first equation of \eqref{E:FD_e2_Eq} is independent of $r$, we substitute \eqref{E:FD_e2_Eq} for the above equality and use $g_1-g_0=g$ to get
\begin{align} \label{E:FD_e2_Cond}
  \begin{aligned}
    &g\{k_d^{-1}\partial^\circ\eta_0+k_d^{-2}V_\Gamma^2\eta_0-\Delta_\Gamma\eta_0-k_d^{-1}V_\Gamma H\eta_0-k_d^{-1}f\} \\
    &\qquad = \nabla_\Gamma g\cdot\nabla_\Gamma\eta_0-k_d^{-1}(\partial^\circ g)\eta_0+k_d^{-2}gV_\Gamma^2\eta_0,
  \end{aligned}
\end{align}
and this equation can be rewritten as
\begin{align} \label{E:FD_e0_Lim}
  \partial^\circ(g\eta_0)-gV_\Gamma H\eta_0-k_d\,\mathrm{div}_\Gamma(g\nabla_\Gamma\eta_0) = gf.
\end{align}
Hence we obtain the limit equation \eqref{E:Limit} as the necessary condition on the zeroth order term $\eta_0$ in order that the second order term $\eta_2$ exists.
Moreover, when \eqref{E:FD_e2_Cond} is satisfied, we can rewrite \eqref{E:FD_e2_Eq} as
\begin{align*}
  \left\{
  \begin{aligned}
    \partial_r^2\eta_2(r) &= \frac{1}{g}\{\nabla_\Gamma g\cdot\nabla_\Gamma\eta_0-k_d^{-1}(\partial^\circ g)\eta_0+k_d^{-2}gV_\Gamma^2\eta_0\}, \quad r\in(g_0,g_1), \\
    \partial_r\eta_2(g_i) &= \nabla_\Gamma g_i\cdot\nabla_\Gamma\eta_0-k_d^{-1}(\partial^\circ g_i)\eta_0+k_d^{-2}g_iV_\Gamma^2\eta_0, \quad i=0,1.
  \end{aligned}
  \right.
\end{align*}
Then, setting
\begin{align*}
  \zeta_i = \frac{1}{g}\{\nabla_\Gamma g_i\cdot\nabla_\Gamma\eta_0-k_d^{-1}(\partial^\circ g_i)\eta_0+k_d^{-2}g_iV_\Gamma^2\eta_0\}, \quad i=0,1,
\end{align*}
we see that the above problem is of the form
\begin{align*}
  \partial_r^2\eta_2(r) = \zeta_1-\zeta_0, \quad r\in(g_0,g_1), \quad \partial_r\eta_2(g_i) = g\zeta_i, \quad i=0,1.
\end{align*}
Hence $\partial_r\eta_2(r)=g\zeta_0+(r-g_0)(\zeta_1-\zeta_0)$ for $r\in[g_0,g_1]$ and
\begin{align*}
  \eta_2(r) &= \eta_2(g_0)+(r-g_0)g\zeta_0+\frac{1}{2}(r-g_0)^2(\zeta_1-\zeta_0) \\
  &= \left\{\eta_2(g_0)-g_0g\zeta_0+\frac{1}{2}g_0^2(\zeta_1-\zeta_0)\right\}+r\{g\zeta_0-g_0(\zeta_1-\zeta_0)\}+\frac{1}{2}r^2(\zeta_1-\zeta_0)
\end{align*}
for $r\in[g_0,g_1]$.
Moreover, as in the case of $\eta_1(g_0)$, we may take
\begin{align*}
  \eta_2(g_0) = g_0g\zeta_0-\frac{1}{2}g_0^2(\zeta_1-\zeta_0).
\end{align*}
Hence, noting that $g\zeta_0-g_0(\zeta_1-\zeta_0)=g_1\zeta_0-g_0\zeta_1$, we obtain
\begin{align} \label{E:FD_eta2}
  \eta_2(r) = r(g_1\zeta_0-g_0\zeta_1)+\frac{1}{2}r^2(\zeta_1-\zeta_0), \quad r\in[g_0,g_1].
\end{align}
Finally, we conclude by \eqref{E:FD_Asymp}, \eqref{E:FD_e0_Eq}, \eqref{E:FD_eta1}, \eqref{E:FD_e0_Lim}, and \eqref{E:FD_eta2} that if we set
\begin{align*}
  \rho_\eta^\varepsilon(x,t) = \sum_{k=0}^2\varepsilon^k\eta_k(\pi(x,t),r,\varepsilon^{-1}d(x,t)), \quad (x,t)\in\overline{Q_{\varepsilon,T}},
\end{align*}
where $\eta_0(y,t,r)=\eta_0(y,t)$ is independent of $r$ and satisfies \eqref{E:Limit}, then $\rho_\eta^\varepsilon$ is of the form \eqref{E:PM_ee12} and satisfies \eqref{E:Heat_MTD} approximately of order $\varepsilon$ in $Q_{\varepsilon,T}$ and $\varepsilon^2$ on $\partial_\ell Q_{\varepsilon,T}$.

\bibliographystyle{abbrv}
\bibliography{HMTD_C0}

\begin{thebibliography}{10}

\bibitem{AbBuGa22_QP}
H.~Abels, F.~Bürger, and H.~Garcke.
\newblock Qualitative properties for a system coupling scaled mean curvature
  flow and diffusion.
\newblock {\em arXiv:2205.02493}, 2022.

\bibitem{AbBuGa22_ST}
H.~Abels, F.~Bürger, and H.~Garcke.
\newblock Short time existence for coupling of scaled mean curvature flow and
  diffusion.
\newblock {\em arXiv:2204.07626}, 2022.

\bibitem{AlCaDjEl21}
A.~Alphonse, D.~Caetano, A.~Djurdjevac, and C.~M. Elliott.
\newblock Function spaces, time derivatives and compactness for evolving
  families of {B}anach spaces with applications to {PDE}s.
\newblock {\em arXiv:2105.07908}, 2021.

\bibitem{AlpEll15}
A.~Alphonse and C.~M. Elliott.
\newblock A {S}tefan problem on an evolving surface.
\newblock {\em Philos. Trans. Roy. Soc. A}, 373(2050):20140279, 16, 2015.

\bibitem{AlpEll16}
A.~Alphonse and C.~M. Elliott.
\newblock Well-posedness of a fractional porous medium equation on an evolving
  surface.
\newblock {\em Nonlinear Anal.}, 137:3--42, 2016.

\bibitem{AlElSt15_PM}
A.~Alphonse, C.~M. Elliott, and B.~Stinner.
\newblock An abstract framework for parabolic {PDE}s on evolving spaces.
\newblock {\em Port. Math.}, 72(1):1--46, 2015.

\bibitem{AlElSt15_IFB}
A.~Alphonse, C.~M. Elliott, and B.~Stinner.
\newblock On some linear parabolic {PDE}s on moving hypersurfaces.
\newblock {\em Interfaces Free Bound.}, 17(2):157--187, 2015.

\bibitem{ArrDeS09}
M.~Arroyo and A.~DeSimone.
\newblock Relaxation dynamics of fluid membranes.
\newblock {\em Phys. Rev. E (3)}, 79(3):031915, 17, 2009.

\bibitem{CaeEll21}
D.~Caetano and C.~M. Elliott.
\newblock Cahn-{H}illiard equations on an evolving surface.
\newblock {\em European J. Appl. Math.}, 32(5):937--1000, 2021.

\bibitem{DeElMiSt19}
K.~Deckelnick, C.~M. Elliott, T.-H. Miura, and V.~Styles.
\newblock Hamilton-{J}acobi equations on an evolving surface.
\newblock {\em Math. Comp.}, 88(320):2635--2664, 2019.

\bibitem{DecSty22}
K.~Deckelnick and V.~Styles.
\newblock Finite element error analysis for a system coupling surface evolution
  to diffusion on the surface.
\newblock {\em Interfaces Free Bound.}, 24(1):63--93, 2022.

\bibitem{DziEll07}
G.~Dziuk and C.~M. Elliott.
\newblock Finite elements on evolving surfaces.
\newblock {\em IMA J. Numer. Anal.}, 27(2):262--292, 2007.

\bibitem{DziEll13_AN}
G.~Dziuk and C.~M. Elliott.
\newblock Finite element methods for surface {PDE}s.
\newblock {\em Acta Numer.}, 22:289--396, 2013.

\bibitem{DziEll13_MC}
G.~Dziuk and C.~M. Elliott.
\newblock {$L^2$}-estimates for the evolving surface finite element method.
\newblock {\em Math. Comp.}, 82(281):1--24, 2013.

\bibitem{EllFri15}
C.~M. Elliott and H.~Fritz.
\newblock Time-periodic solutions of advection-diffusion equations on moving
  hypersurfaces.
\newblock {\em SIAM J. Math. Anal.}, 47(3):1693--1718, 2015.

\bibitem{ElGaKo22}
C.~M. Elliott, H.~Garcke, and B.~Kov\'{a}cs.
\newblock Numerical analysis for the interaction of mean curvature flow and
  diffusion on closed surfaces.
\newblock {\em Numer. Math.}, 151(4):873--925, 2022.

\bibitem{EllRan15}
C.~M. Elliott and T.~Ranner.
\newblock Evolving surface finite element method for the {C}ahn-{H}illiard
  equation.
\newblock {\em Numer. Math.}, 129(3):483--534, 2015.

\bibitem{EllSti09}
C.~M. Elliott and B.~Stinner.
\newblock Analysis of a diffuse interface approach to an advection diffusion
  equation on a moving surface.
\newblock {\em Math. Models Methods Appl. Sci.}, 19(5):787--802, 2009.

\bibitem{ElStStWe11}
C.~M. Elliott, B.~Stinner, V.~Styles, and R.~Welford.
\newblock Numerical computation of advection and diffusion on evolving diffuse
  interfaces.
\newblock {\em IMA J. Numer. Anal.}, 31(3):786--812, 2011.

\bibitem{GilTru01}
D.~Gilbarg and N.~S. Trudinger.
\newblock {\em Elliptic partial differential equations of second order}.
\newblock Classics in Mathematics. Springer-Verlag, Berlin, 2001.
\newblock Reprint of the 1998 edition.

\bibitem{HalRau92_DH}
J.~K. Hale and G.~Raugel.
\newblock A damped hyperbolic equation on thin domains.
\newblock {\em Trans. Amer. Math. Soc.}, 329(1):185--219, 1992.

\bibitem{HalRau92_RD}
J.~K. Hale and G.~Raugel.
\newblock Reaction-diffusion equation on thin domains.
\newblock {\em J. Math. Pures Appl. (9)}, 71(1):33--95, 1992.

\bibitem{HalRau95}
J.~K. Hale and G.~Raugel.
\newblock A reaction-diffusion equation on a thin {$L$}-shaped domain.
\newblock {\em Proc. Roy. Soc. Edinburgh Sect. A}, 125(2):283--327, 1995.

\bibitem{JaOlRe18}
T.~Jankuhn, M.~A. Olshanskii, and A.~Reusken.
\newblock Incompressible fluid problems on embedded surfaces: modeling and
  variational formulations.
\newblock {\em Interfaces Free Bound.}, 20(3):353--377, 2018.

\bibitem{JimKur16}
S.~Jimbo and K.~Kurata.
\newblock Asymptotic behavior of eigenvalues of the {L}aplacian on a thin
  domain under the mixed boundary condition.
\newblock {\em Indiana Univ. Math. J.}, 65(3):867--898, 2016.

\bibitem{KoLiGi17}
H.~Koba, C.~Liu, and Y.~Giga.
\newblock Energetic variational approaches for incompressible fluid systems on
  an evolving surface.
\newblock {\em Quart. Appl. Math.}, 75(2):359--389, 2017.

\bibitem{Kos00}
S.~Kosugi.
\newblock A semilinear elliptic equation in a thin network-shaped domain.
\newblock {\em J. Math. Soc. Japan}, 52(3):673--697, 2000.

\bibitem{Kos02}
S.~Kosugi.
\newblock Semilinear elliptic equations on thin network-shaped domains with
  variable thickness.
\newblock {\em J. Differential Equations}, 183(1):165--188, 2002.

\bibitem{Kre14}
D.~Krej\v{c}i\v{r}\'{\i}k.
\newblock Spectrum of the {L}aplacian in narrow tubular neighbourhoods of
  hypersurfaces with combined {D}irichlet and {N}eumann boundary conditions.
\newblock {\em Math. Bohem.}, 139(2):185--193, 2014.

\bibitem{LaSoUr68}
O.~A. Lady\v{z}enskaja, V.~A. Solonnikov, and N.~N. Ural$'$ceva.
\newblock {\em Linear and quasilinear equations of parabolic type}.
\newblock Translations of Mathematical Monographs, Vol. 23. American
  Mathematical Society, Providence, R.I., 1968.
\newblock Translated from the Russian by S. Smith.

\bibitem{Lee18}
J.~M. Lee.
\newblock {\em Introduction to {R}iemannian manifolds}, volume 176 of {\em
  Graduate Texts in Mathematics}.
\newblock Springer, Cham, 2018.
\newblock Second edition of [ MR1468735].

\bibitem{Lie96}
G.~M. Lieberman.
\newblock {\em Second order parabolic differential equations}.
\newblock World Scientific Publishing Co., Inc., River Edge, NJ, 1996.

\bibitem{MaNaPl00_02}
V.~Maz$'$ya, S.~Nazarov, and B.~Plamenevskij.
\newblock {\em Asymptotic theory of elliptic boundary value problems in
  singularly perturbed domains. {V}ol. {II}}, volume 112 of {\em Operator
  Theory: Advances and Applications}.
\newblock Birkh\"{a}user Verlag, Basel, 2000.
\newblock Translated from the German by Plamenevskij.

\bibitem{Miu17}
T.-H. Miura.
\newblock Zero width limit of the heat equation on moving thin domains.
\newblock {\em Interfaces Free Bound.}, 19(1):31--77, 2017.

\bibitem{Miu18}
T.-H. Miura.
\newblock On singular limit equations for incompressible fluids in moving thin
  domains.
\newblock {\em Quart. Appl. Math.}, 76(2):215--251, 2018.

\bibitem{Miu20_03}
T.-H. Miura.
\newblock Navier-{S}tokes equations in a curved thin domain, {P}art {III}:
  thin-film limit.
\newblock {\em Adv. Differential Equations}, 25(9-10):457--626, 2020.

\bibitem{Miu21_02}
T.-H. Miura.
\newblock Navier-{S}tokes {E}quations in a {C}urved {T}hin {D}omain, {P}art
  {II}: {G}lobal {E}xistence of a {S}trong {S}olution.
\newblock {\em J. Math. Fluid Mech.}, 23(1):Paper No. 7, 60, 2021.

\bibitem{Miu22_01}
T.-H. Miura.
\newblock Navier-{S}tokes {E}quations in a {C}urved {T}hin {D}omain, {P}art
  {I}: {U}niform {E}stimates for the {S}tokes {O}perator.
\newblock {\em J. Math. Sci. Univ. Tokyo}, 29(2):149--256, 2022.

\bibitem{MiGiLi18}
T.-H. Miura, Y.~Giga, and C.~Liu.
\newblock An energetic variational approach for nonlinear diffusion equations
  in moving thin domains.
\newblock {\em Adv. Math. Sci. Appl.}, 27(1):115--141, 2018.

\bibitem{NiReVo19}
I.~Nitschke, S.~Reuther, and A.~Voigt.
\newblock Hydrodynamic interactions in polar liquid crystals on evolving
  surfaces.
\newblock {\em Phys. Rev. Fluids}, 4:044002, Apr 2019.

\bibitem{NiReVo20}
I.~Nitschke, S.~Reuther, and A.~Voigt.
\newblock Liquid crystals on deformable surfaces.
\newblock {\em Proc. A.}, 476(2241):20200313, 23, 2020.

\bibitem{OCoSti16}
D.~O'Connor and B.~Stinner.
\newblock The {C}ahn-{H}illiard equation on an evolving surface.
\newblock {\em arXiv:1607.05627}, 2016.

\bibitem{OlReXu14}
M.~A. Olshanskii, A.~Reusken, and X.~Xu.
\newblock An {E}ulerian space-time finite element method for diffusion problems
  on evolving surfaces.
\newblock {\em SIAM J. Numer. Anal.}, 52(3):1354--1377, 2014.

\bibitem{OlReZh22}
M.~A. Olshanskii, A.~Reusken, and A.~Zhiliakov.
\newblock Tangential {N}avier-{S}tokes equations on evolving surfaces: Analysis
  and simulations.
\newblock {\em arXiv:2203.01521}, 2022.

\bibitem{PerSil13}
J.~V. Pereira and R.~P. Silva.
\newblock Reaction-diffusion equations in a noncylindrical thin domain.
\newblock {\em Bound. Value Probl.}, pages 2013:248, 10, 2013.

\bibitem{PrRiRy02}
M.~Prizzi, M.~Rinaldi, and K.~P. Rybakowski.
\newblock Curved thin domains and parabolic equations.
\newblock {\em Studia Math.}, 151(2):109--140, 2002.

\bibitem{PriRyb01}
M.~Prizzi and K.~P. Rybakowski.
\newblock The effect of domain squeezing upon the dynamics of
  reaction-diffusion equations.
\newblock {\em J. Differential Equations}, 173(2):271--320, 2001.

\bibitem{PriRyb03}
M.~Prizzi and K.~P. Rybakowski.
\newblock On inertial manifolds for reaction-diffusion equations on genuinely
  high-dimensional thin domains.
\newblock {\em Studia Math.}, 154(3):253--275, 2003.

\bibitem{Rau95}
G.~Raugel.
\newblock Dynamics of partial differential equations on thin domains.
\newblock In {\em Dynamical systems ({M}ontecatini {T}erme, 1994)}, volume 1609
  of {\em Lecture Notes in Math.}, pages 208--315. Springer, Berlin, 1995.

\bibitem{SaYaSaMa20}
A.~Sahu, Y.~A.~D. Omar, R.~A. Sauer, and K.~K. Mandadapu.
\newblock Arbitrary {L}agrangian-{E}ulerian finite element method for curved
  and deforming surfaces: {I}. {G}eneral theory and application to fluid
  interfaces.
\newblock {\em J. Comput. Phys.}, 407:109253, 49, 2020.

\bibitem{Sch96}
M.~Schatzman.
\newblock On the eigenvalues of the {L}aplace operator on a thin set with
  {N}eumann boundary conditions.
\newblock {\em Appl. Anal.}, 61(3-4):293--306, 1996.

\bibitem{TemZia97}
R.~Temam and M.~Ziane.
\newblock Navier-{S}tokes equations in thin spherical domains.
\newblock In {\em Optimization methods in partial differential equations
  ({S}outh {H}adley, {MA}, 1996)}, volume 209 of {\em Contemp. Math.}, pages
  281--314. Amer. Math. Soc., Providence, RI, 1997.

\bibitem{Vie14}
M.~Vierling.
\newblock Parabolic optimal control problems on evolving surfaces subject to
  point-wise box constraints on the control---theory and numerical realization.
\newblock {\em Interfaces Free Bound.}, 16(2):137--173, 2014.

\bibitem{Yac18}
T.~Yachimura.
\newblock Two-phase eigenvalue problem on thin domains with {N}eumann boundary
  condition.
\newblock {\em Differential Integral Equations}, 31(9-10):735--760, 2018.

\bibitem{Yan90}
E.~Yanagida.
\newblock Existence of stable stationary solutions of scalar reaction-diffusion
  equations in thin tubular domains.
\newblock {\em Appl. Anal.}, 36(3-4):171--188, 1990.

\bibitem{YaOzSa16}
A.~Yavari, A.~Ozakin, and S.~Sadik.
\newblock Nonlinear elasticity in a deforming ambient space.
\newblock {\em J. Nonlinear Sci.}, 26(6):1651--1692, 2016.

\end{thebibliography}

\end{document}